\documentclass[a4paper,12pt]{amsart}
\usepackage[a4paper,hscale=0.7,vscale=0.75,centering]{geometry}
\usepackage{fullpage}
\usepackage{url}
\usepackage{bbm}
\usepackage{dsfont}

\usepackage{amsthm,amsmath,amsfonts,stmaryrd,bbm,hyperref,geometry,color,amssymb}

\newtheorem{theorem}{Theorem}[section]

\newtheorem{corollary}[theorem]{Corollary}
\newtheorem{proposition}[theorem]{Proposition}

\theoremstyle{definition}

\newtheorem{example}[theorem]{Example}

\newtheorem{assumption}{Assumption}

\newcommand{\po}{\left(}
\newcommand{\pf}{\right)}

\newcommand{\R}{\mathbb R}

\newcommand{\dd}{\text{d}}

\theoremstyle{remark}
\newtheorem{remark}[theorem]{Remark}

\numberwithin{equation}{section}

\title{$L^2$-Wasserstein contraction for Euler schemes of elliptic diffusions and interacting particle systems}

\author{Linshan Liu}
\address{School of Mathematical and Computer Sciences, Heriot-Watt University, Edinburgh, EH14 4AS, UK}
\email{ll2018@hw.ac.uk}

\author{Mateusz B. Majka}
\address{School of Mathematical and Computer Sciences, Heriot-Watt University, Edinburgh, EH14 4AS, UK}
\email{m.majka@hw.ac.uk}

\author{Pierre Monmarch\'{e}}
\address{Laboratoire Jacques-Louis Lions and Laboratoire de Chimie Th\'eorique, Sorbonne Universit\'{e}, Paris, 4 place Jussieu
75005, France}
\email{pierre.monmarche@sorbonne-universite.fr}

\begin{document}
	
	\begin{abstract}
		We show the $L^2$-Wasserstein contraction for the transition kernel of a discretised diffusion process, under a contractivity at infinity condition on the drift and a sufficiently high diffusivity requirement. This extends recent results that, under similar assumptions on the drift but without the diffusivity restrictions, showed the $L^1$-Wasserstein contraction, or $L^p$-Wasserstein bounds for $p > 1$ that were, however, not true contractions. We explain how showing the true $L^2$-Wasserstein contraction is crucial for obtaining the local Poincar\'{e} inequality for the transition kernel of the Euler scheme of a diffusion. Moreover, we discuss other consequences of our contraction results, such as concentration inequalities and convergence rates in KL-divergence and total variation. We also study the corresponding $L^2$-Wasserstein contraction for discretisations of interacting diffusions. As a particular application, this allows us to analyse the behaviour of particle systems that can be used to approximate a class of McKean-Vlasov SDEs that were recently studied in the mean-field optimization literature.
	\end{abstract}

	\maketitle
 \section{Introduction}
Let $Q(x, dy)$ be a Markov transition kernel on $\mathbb{R}^d$. We say that $Q$ is contractive with respect to the $L^p$-Wasserstein distance $\mathcal{W}_p$ (whose definition is recalled in \eqref{eq:defWasserstein} below), if there exists a constant $\theta \in (0,1)$ such that for any probability measures $\mu$ and $\nu$ on $\mathbb{R}^d$,
\begin{equation}\label{eq:contraction}
\mathcal{W}_p(\mu Q, \nu Q) \leq (1-\theta) \mathcal{W}_p(\mu , \nu) \,.
\end{equation}
 Contractions of this type have been intensively studied in recent years, see e.g.\ \cite{EberleMajka2019,QinHobert2022,RudolfSchweizer2018,JoulinOllivier2010,Ollivier2009,HuangMajkaWang2022,MajkaMijatovicSzpruch2020} and the references therein. The particular focus has been on the case of $p=1$, which was originally considered by Dobrushin in \cite{Dobrushin1970}, and recently received increased attention due to its connection with the notion of the Ricci curvature \cite{Ollivier2009} and concentration and error bounds for Markov Chain Monte Carlo algorithms \cite{JoulinOllivier2010}.
 
 In particular, condition \eqref{eq:contraction} with $Q = P_t$ for $t \geq 0$ has been studied in detail in the case where $P_t$ is the transition kernel of a $d$-dimensional diffusion process 
 \begin{equation}\label{eq:diffusion}
 dX_t = b(X_t)dt + \sqrt{2T} dW_t \,,
 \end{equation}
 where $b: \mathbb{R}^d \to \mathbb{R}^d$ is the drift, $(W_t)_{t \geq 0}$ is a standard Brownian motion in $\mathbb{R}^d$, and $T>0$ is the diffusivity parameter. 
 It is known (see \cite[Corollary 1.4]{vonRenesseSturm2005} in the reversible case, \cite{Pierre2} in the general case) that for any $p \geq 1$, the strict exponential contractivity of $P_t$, in the sense that $\mathcal{W}_p(\mu P_t, \nu P_t) \leq e^{-ct} \mathcal{W}_p(\mu,\nu)$ holds for all $t \geq 0$ and all $\mu$, $\nu \in \mathcal{P}(\mathbb{R}^d)$, with some constant $c>0$, is equivalent to the contractivity condition on the drift with the same constant $c>0$, in the sense that
 \begin{equation}\label{eq:driftContractivity}
 \langle b(x) - b(y) , x -y \rangle \leq -c |x-y|^2
 \end{equation}
 holds for all $x$, $y \in \mathbb{R}^d$. However, in the case of $p=1$, even for non-reversible diffusions, by utilising the coupling approach, it has been possible to obtain a weak exponential contractivity condition $\mathcal{W}_1(\mu P_t, \nu P_t) \leq Me^{-ct} \mathcal{W}_1(\mu , \nu)$ with constant $M>1$ and $c>0$, under a relaxed condition on the drift where \eqref{eq:driftContractivity} is required to hold only for $x$, $y \in \mathbb{R}^d$ such that $|x-y| > R$, for an arbitrarily large $R>0$, see \cite{Eberle2016,EberleGuillinZimmer2019}. Subsequently, the coupling approach was also used to obtain exponential upper bounds for $p > 1$ that are, however, not true contractions, in the sense that $\mathcal{W}_p(\mu P_t, \nu P_t) \leq Me^{-ct} f(\mu,\nu)$, where $f$ is a function of $\mu$ and $\nu$ that cannot be controlled from above by $\mathcal{W}_p(\mu , \nu)$, see \cite{LuoWang2016}. In \cite{moiCourbureTemperature}, the third author of the present paper  obtained true $L^p$-Wasserstein contractions for $p>1$ under similar assumptions on the drift, and an additional assumption of sufficiently large value of $T$ in \eqref{eq:diffusion}. As discussed in \cite{moiCourbureTemperature}, out of all $p>1$, of particular interest is the case $p=2$, due to its connection with the Poincar\'{e} inequality and the KL-divergence (see also Sections \ref{sec: proof_local_Poincare} and \ref{sec:EntropyRegularization} in the present paper).
 
 Another class of processes for which Wasserstein contractions have been studied in detail, are discretised versions of diffusions and other stochastic differential equations \cite{EberleMajka2019, HuangMajkaWang2022, MajkaMijatovicSzpruch2020}. Namely, let $Q$ be the transition kernel of a Markov chain in $\mathbb{R}^d$ defined recursively by
 \begin{equation}\label{eq:chain}
 X_{k+1} = X_k + \delta b(X_k) + \sqrt{2T\delta} Z_{k+1} \,,
 \end{equation}
 where $\delta > 0$ is the discretisation parameter and $(Z_k)_{k=1}^{\infty}$ are i.i.d.\ random variables with the standard normal distribution. 
 In this setting, $L^1$-Wasserstein contractions of the form $\mathcal{W}_{1}(\mu Q^n, \nu Q^n) \leq M(1-c)^n \mathcal{W}_{1}(\mu , \nu)$ for any $n \geq 1$, with constant $M >1$, $c>0$, have been obtained in \cite{EberleMajka2019} under the contractivity at infinity condition on the drift (i.e., with \eqref{eq:driftContractivity} holding for $|x-y|>R$), whereas upper bounds on $L^2$-Wasserstein distances that are not true contractions were considered in \cite{MajkaMijatovicSzpruch2020}. Most recently, \cite{HuangMajkaWang2022} extended those results to other $p > 1$ and other types of noise in \eqref{eq:chain}, not necessarily Gaussian.
 
 The goal of the present paper is to adapt the techniques from \cite{moiCourbureTemperature} to the setting of \eqref{eq:chain}, in order to go beyond the results from \cite{EberleMajka2019, MajkaMijatovicSzpruch2020, HuangMajkaWang2022} and obtain true $L^2$-Wasserstein contractions, under a contractivity at infinity assumption on the drift in \eqref{eq:chain} and an additional requirement of sufficiently high value of $T$. We will also discuss several consequences of such results, including the local Poincar\'{e} inequality for $Q^n$ for all $n \geq 1$ (which, remarkably, can be obtained with a constant independent of the discretisation parameter $\delta$), concentration inequalities, entropy-cost regularization inequalities and bounds in $\operatorname{KL}$-divergence and total variation. Additionally, in contrast to \cite{EberleMajka2019, MajkaMijatovicSzpruch2020, HuangMajkaWang2022}, we will also study in detail contractions of discretised interacting diffusions. More precisely, we will consider a system of $N$ chains
 \begin{equation}\label{eq:particleSystem}
 X^i_{k+1} = X^i_k + \delta F(X^i_k) +\delta G_i(\textbf{X}_k)+ \sqrt{2\delta T} Z^i_k \,,
 \end{equation}
where the drift consists of the confinement component $F: \mathbb{R}^d \to \mathbb{R}^d$, and the interaction components $G_i: \mathbb{R}^{Nd} \to \mathbb{R}^d$ for $i \in \{ 1, \ldots , N \}$, $\delta > 0$ is the discretisation parameter, $T>0$ is the diffusivity parameter,  $(Z^i_k)_{i \in \{ 1, \ldots, N \}, k \geq 1}$ are i.i.d.\ standard normal random variables in $\mathbb{R}^d$ and we denote $\textbf{X}_k=(X_k^1,\dots,X_k^N)$.

\begin{example}
One of our motivations for studying such systems is their connection to a class of McKean-Vlasov SDEs that have recently been studied in the context of mean-field optimization \cite{HuRenSiskaSzpruch2021, Nitanda2022, Chizat2022, Suzuki2023} and the mean-field analysis of two-player zero-sum games \cite{Domingo-Enrich2020, YulongLu2023}. In particular, \cite{Domingo-Enrich2020} considered the problem of finding mixed Nash equilibria of games with entropy-regularised payoff functions $\mathcal{L}(\mu,\nu) := \int \int \ell(x,y) \mu(dx) \nu(dy) + \beta^{-1}(H(\nu) - H(\mu))$, for some integrable function $\ell: \mathbb{R}^d \times \mathbb{R}^d \to \mathbb{R}$, where $\beta > 0$ is the regularisation parameter and $H$ is the entropy, by utilising a gradient flow $(\mu_x, \mu_y) = (\mu_x(t), \mu_y(t))_{t \geq 0}$ of probability measures on $\mathbb{R}^{2d}$ given by
\begin{equation*}
\begin{cases}
\partial_t \mu_x &= \nabla_x \cdot \left( \mu_x \nabla_x \int \ell(x,y) \mu_y(dy) \right) + \beta^{-1} \Delta_x \mu_x \\
\partial_t \mu_y &= -\nabla_y \cdot \left( \mu_y \nabla_y \int \ell(x,y) \mu_x(dx) \right) + \beta^{-1} \Delta_y \mu_y 
\end{cases} \,.
\end{equation*}
Note that this corresponds (by taking $\mu_x = \mu_x(t,x)$ as the density of $X_t$ and $\mu_y = \mu_y(t,y)$ as the density of $Y_t$, respectively) to the $2d$-dimensional McKean-Vlasov SDE
\begin{equation*}
\begin{cases}
dX_t = - \int \ell (X_t,y) \operatorname{Law}(Y_t)(dy) dt + \sqrt{2}\beta^{-1} dW_t^1 \\
dY_t = \int \ell (x,Y_t) \operatorname{Law}(X_t)(dy) dt + \sqrt{2}\beta^{-1} dW_t^2 \,,
\end{cases}
\end{equation*}
where $(W_t^1,W_t^2)_{t \geq 0}$ is the standard Brownian motion in $\mathbb{R}^{2d}$.
As shown in \cite{Domingo-Enrich2020}, this flow can be approximated by a system of $2N$ particles given by \eqref{eq:particleSystem}, with $T=\beta^{-2}$, $F= 0$ and for $x = (x_1, \ldots, x_N)$, $y = (y_{N+1}, \ldots, y_{2N})$ we have $G_i(x,y) =-  \frac{1}{N} \sum_{j=N+1}^{2N} \ell(x_i,y_j)$ for $i \in \{ 1, \ldots, N \}$ and $G_i(x,y) =\frac{1}{N} \sum_{j=1}^{N}\ell(x_j,y_i)$  for $i \in \{ N+1, \ldots, 2N \}$, see (4) in \cite{Domingo-Enrich2020}. If we regularise the payoff function $\mathcal{L}$ by $\operatorname{KL}(\cdot|\pi)$ with $\pi \propto e^{-U}$ instead of $H$ (see \cite{HuRenSiskaSzpruch2021,RenWang2022,LiuMajkaSzpruch2023}), we obtain an additional drift term $F = -\nabla U$. Under appropriate assumptions on $U$, this falls into the framework of our results for discretised interacting particle systems (see Theorem \ref{thm:rhoContractionForSystem} below), and allows us to show an $L^2$-Wasserstein contraction for sufficiently small $\beta$ (large $T$).
\end{example}

The remaining part of the paper is organised as follows. In Section \ref{section:main} we formulate our main results: a contraction for the Euler scheme of a single diffusion (Theorem \ref{thm:rhoContraction}), a contraction for the Euler scheme of an interacting particle system (Theorem \ref{thm:rhoContractionForSystem}) and the Poincar\'{e} inequality for $k$-step transition kernels for the Euler scheme of a single diffusion (Theorem \ref{thm:PoincareInequality}). The proofs of these results are presented in Sections \ref{sec: proofForOneParticle}, \ref{sec:proofInteracting} and \ref{sec: proof_local_Poincare}, respectively, with an auxiliary Section \ref{sec:existenceKappa} discussing the construction of the distance function used in the main proofs. We then conclude with two additional sections, discussing applications of our results to concentration inequalities (Section \ref{sec:ConfidenceIntervals}) and convergence rates in KL-divergence using an entropy-cost regularization inequality (Section \ref{sec:EntropyRegularization}).

 \section{Main results}\label{section:main}
We now present the notation and terminology that we use throughout the paper. By $\mathcal{P}(\mathbb{R}^d)$ we denote the space of probability measures over $\mathbb{R}^d$. For $\mu$, $\nu \in \mathcal{P}(\mathbb{R}^d)$ and for $p\in[1,\infty)$, the $L^p$-Wasserstein distance $\mathcal{W}_{p,\rho}$ associated to a cost function $\rho: \mathbb{R}^d\times \mathbb{R}^d \to \mathbb{R}_+$ is defined as 
\begin{equation}\label{eq:defWasserstein}
\mathcal{W}_{p,\rho}(\mu,\nu):= \left(\inf_{\pi \in \Pi(\mu,\nu)} \int_{\mathbb{R}^d\times \mathbb{R}^d} \rho(x,y)^p d \pi(x,y) \right)^{1/p} \,,  
\end{equation}
where $\Pi(\mu,\nu)$ is the set of all couplings between $\mu$ and $\nu$. If $Q$ is a transition kernel, we say that $Q$ satisfies a (true) $\mathcal{W}_{p,\rho}$-contraction if there exists $\theta \in (0,1)$ such that
\begin{equation}\label{eq: Kantorovich_contraction_1}
    \mathcal{W}_{p,\rho}(\mu Q, \nu Q) \leq (1-\theta) \mathcal{W}_{p,\rho}(\mu , \nu ), \quad \text{ for all } \mu,\nu \in \mathcal{P}(\mathbb{R}^d),
\end{equation}
which is of course equivalent to 
\begin{equation}\label{eq: Kantorovich_contraction_2}
    \mathcal{W}_{p,\rho} \left(\mu Q^k, \nu Q^k \right) \leq (1-\theta)^k \mathcal{W}_{p,\rho}(\mu , \nu ), \quad \text{ for all } \mu,\nu \in \mathcal{P}(\mathbb{R}^d) \quad \text{ and for all } k \in \mathbb{N}.
\end{equation}
We say that $Q$ satisfies a weak $\mathcal{W}_{p,\rho}$-contraction if there exist $M \geq 1 >\theta>0 $ such that
\begin{equation}\label{eq: weak_K_contraction}
\mathcal{W}_{p,\rho} \left(\mu Q^k, \nu Q^k \right) \leq M(1-\theta)^k \mathcal{W}_{p,\rho}(\mu , \nu ), \quad \text{ for all } \mu,\nu \in \mathcal{P}(\mathbb{R}^d) \quad \text{ and for all } k \in \mathbb{N}.
\end{equation}
 When $\rho$ is the Euclidean distance on $\R^d$ (denoted $\lvert \cdot \rvert$), we denote $\mathcal{W}_{p,\rho}$ simply by $\mathcal{W}_p$ and we call the above conditions $L^p$-Wasserstein contractions.

\subsection{Contraction for Euler scheme of a single diffusion at high diffusivity}\label{subsec:SingleEuler}
We consider the Euler scheme \eqref{eq:chain}, under the following assumptions on $b$.
\begin{assumption}\label{hypo:1}
Suppose $b: \mathbb{R}^d \to \mathbb{R}^d$ is a Lipschitz function with a Lipschitz constant $L_b>0$ in the sense that 
\begin{equation}
    \lvert b(x) - b(y)\rvert \leq L_b \lvert x -y\rvert \quad \text{ for all } x, y \in \mathbb{R}^d\,.
\end{equation}
\end{assumption}

\begin{assumption}\label{hypo:2}
Suppose there exist $R,c,K>0$ such that
    \begin{equation}
        \langle x-y, b(x)-b(y) \rangle \leq 
        \begin{cases}
            -c \lvert x-y \rvert^2  & \text{if }|x|\geqslant R\text{ or }|y|\geqslant R \\
            K \lvert x-y \rvert^2  & \text{otherwise}
        \end{cases}.
    \end{equation}
\end{assumption}
Our ultimate objective in this setting is to demonstrate an $L^2$-Wasserstein contraction for the Markov chain \eqref{eq:chain}. To this end, we initially establish a $\mathcal{W}_{1,\rho}$-contraction for a carefully chosen  pseudo-distance function $\rho$, which we can subsequently use to retrieve the desired $L^2$-contraction. 
In order to construct a suitable $\rho$, we will need the following auxiliary result.

\begin{proposition}\label{prop:existence_kappa}
Suppose Assumption \ref{hypo:2} is satisfied with some positive constants $c$, $R$ and $K$. There exists $\delta_4>0$ (given in \eqref{eq:delta1234} below) such that, for all $\delta\in(0,\delta_4]$, the following holds. Let $Z$ be a random variable with a standard normal distribution on $\mathbb{R}^d$. For any real numbers $a\in [12K,\infty)$ and $L \in (0,\frac{c}{6}]$, we can construct a function $\kappa: \mathbb{R}^d \to \mathbb{R}$ such that: 
 \begin{enumerate}
 \item The function $\kappa$ is non-negative, bounded, and continuously differentiable.
\item There exists $R_*>R$ such that 
\begin{equation*}
    \kappa(x) = 0 \text{ and } \nabla \kappa(x) =0, \quad \forall x \in \mathbb{R}^d \text{ with }\lvert x \rvert \geq R_*
\end{equation*}
\item For all $x\in\R^d $ with $\lvert x\rvert \leq R$ we have 
\begin{equation}\label{eq:condition_bound_a}
    \mathbb{E}\left[\kappa (x+ \sqrt{2 \delta T}Z)\right] \leq \kappa(x) - a \delta T.
\end{equation}
\item For all $x\in\R^d$ we have 
\begin{equation}\label{eq:condition_bound_L}
    \mathbb{E}\left[\kappa (x+ \sqrt{2 \delta T}Z)\right] \leq \kappa(x) + L\delta T.
\end{equation}
\end{enumerate}
\end{proposition}
The proof of this proposition can be found in Section \ref{sec:existenceKappa}. 
\begin{theorem}\label{thm:rhoContraction} Suppose Assumptions \ref{hypo:1} and \ref{hypo:2} are satisfied. Define $\rho$ as 
\begin{equation}\label{eq: rho}
    \rho(x,y):= \lvert x-y \rvert ^2 (T +\kappa(x) + \kappa(y)) \quad \text{for all } x,y \in \mathbb{R}^d \,,
\end{equation}
with a function $\kappa: \mathbb{R}^d \to \mathbb{R}_+$ as in Proposition \ref{prop:existence_kappa} for some $a\geq 12K$ and $L\leq c/6$. Then there exist $h$, $\delta_0$ and $T_0 > 0$ such that for all $\delta \in (0, \delta_0]$ and all $T \geq T_0$, the Markov kernel $Q$ of the Euler scheme \eqref{eq:chain} satisfies  
 \begin{equation}\label{eq:rhoContraction}
 \mathcal{W}_{1,\rho}(\mu Q, \nu Q) \leq (1-h\delta ) \mathcal{W}_{1,\rho}(\mu , \nu ) \quad \text{ for all } \mu,\nu \in \mathcal{P}(\mathbb{R}^d).
\end{equation}
\end{theorem}

This is proven in Section~\ref{sec: proofForOneParticle}. Explicit expressions are provided for $\delta_0$ in \eqref{eq: delta0}, $T_0$ in \eqref{eq: T_0} and   $h$ in \eqref{eq:h} in Section~\ref{sec: proofForOneParticle} (these expressions involve $\kappa$ and can be bounded in terms of the parameters of the problem only using \eqref{eq:bound_kappa}). The following corollary, also proven in Section~\ref{sec: proofForOneParticle}, shows that Theorem \ref{thm:rhoContraction} implies a weak $L^2$-Wasserstein contraction.
\begin{corollary}\label{cor: WContraction}
Suppose Assumptions \ref{hypo:1} and \ref{hypo:2} hold. Then there exist $h$, $\delta_0$, $T_0 > 0$ and $M \geq 1$, such that for all $\delta \in (0, \delta_0]$ and all $T \geq T_0$, the Markov kernel $Q$ of the Euler scheme \eqref{eq:chain} satisfies, for any $k \geq 1$,  
\begin{equation}\label{eq:WContraction}
    \mathcal{W}_2(\mu Q^k, \nu Q^k) \leq M (1-h \delta)^k \mathcal{W}_2(\mu,\nu), \quad \text{ for all } \mu,\nu \in \mathcal{P}(\mathbb{R}^d).
\end{equation}
\end{corollary}
We give an explicit expression for $M$ in \eqref{eq: explicit_M} ($\delta_0,T_0$ and $h$ are as in Theorem~\ref{thm:rhoContraction}).

\subsection{Contraction for interacting particle system at high diffusivity.}
In this subsection, our objective is to establish contraction properties for interacting particle systems. The specific type of system that we are investigating is given by:
\begin{equation}\label{eq: descrete_ststem}
    \forall i \in \llbracket 1,N\rrbracket, \quad X^i_{k+1} = X^i_k + \delta F(X^i_k) +\delta G_i(\textbf{X}_k)+ \sqrt{2\delta T} Z^i_k.
\end{equation}
We assume that the sequence $(Z^i_k)_{{i\in \llbracket 1,N\rrbracket, k \in \mathbb{N} }}$ comprises independent standard $d$-dimensional Gaussian vectors, and that the process $(\textbf{X}_k)_{k \geq 0}$ corresponds to the dynamics of the interacting particle system $(X^1_k,...,X^N_k)_{k \geq 0}$. We denote the transition kernel of the system \eqref{eq: descrete_ststem} by $\tilde{Q}$, and note that $\textbf{x}$ and $\textbf{y}$ are points in $\mathbb{R}^{dN}$ that can represent the states of the interacting particle system. Moreover, for any $i \in \llbracket 1,N\rrbracket$, $x^{(i)}$ and $y^{(i)}$ are points in $\mathbb{R}^d$ representing the position of the $i^{th}$ particle, i.e., $\textbf{x}:=(x^{(1)},...,x^{(N)})$ and $\textbf{y}:=(y^{(1)},...,y^{(N)})$. The following assumptions are used throughout this section.
\begin{assumption}\label{hypo: F}
    We assume that $F: \mathbb{R}^d \to \mathbb{R}^d $ is a Lipschitz function with Lipschitz constant $L_F>0$ in the sense that
    \begin{equation}
        \lvert F(x) - F(y)\rvert \leq L_F \lvert x -y\rvert \quad \text{ for all } x, y \in \mathbb{R}^d\,.
    \end{equation}
       Furthermore, we assume that there exist $R,c,K>0$ such that
    \begin{equation}
        \langle x-y, F(x)-F(y) \rangle \leq 
        \begin{cases}
            -c \lvert x-y \rvert^2  & \text{if }|x|\geqslant R\text{ or }|y|\geqslant R \\
            K \lvert x-y \rvert^2  & \text{otherwise.}
        \end{cases}   
    \end{equation}
\end{assumption}

\begin{assumption}\label{hypo: G}
   We denote $G:= (G_1,...,G_N)$. We assume that there exist $L_G$, $C_G$, $M_G>0$ and $p \in [1,\infty)$ such that, for all $\textbf{x}, \textbf{y} \in \mathbb{R}^{dN}$ and all $i \in \llbracket 1,N\rrbracket$,
   \begin{align}
   \lvert G(\textbf{x})- G(\textbf{y}) \rvert & \leq L_G \lvert \textbf{x} - \textbf{y}\rvert \label{eq: hypG1}\\
   \sum_{j=1}^N \lvert x^{(j)} - y^{(j)} \rvert \lvert G_j(\textbf{x}) - G_j(\textbf{y})\rvert & \leq C_G\lvert \textbf{x}-\textbf{y}\rvert^2 \label{eq: hypG2}\\
   \lvert G_i (\textbf{x})\rvert & \leq M_G(1+\lvert x^{(i)}\rvert^p). \label{eq: hypG3}
   \end{align}
\end{assumption}
Note that the assumptions on $F$ are the same as assumptions on $b$ in Subsection \ref{subsec:SingleEuler}. Hence we can construct a function $\kappa$ exactly as in Proposition \ref{prop:existence_kappa} (where $R,c,K$ are now the one associated to $F$). We consider such a function in the next statement.

\begin{theorem}\label{thm:rhoContractionForSystem}
     Suppose Assumptions \ref{hypo: F} and \ref{hypo: G} hold. We define a function $\Tilde{\rho}: \mathbb{R}^{dN} \times \mathbb{R}^{dN} \to \mathbb{R}_+$ as
     \begin{equation}\label{eq: defRhoTilde}
        \tilde{\rho} (\textbf{x},\textbf{y}) :=\sum_{i=1}^N \lvert x^{(i)}-y^{(i)}\rvert^2(T + \kappa(x^{(i)}) +\kappa(y^{(i)}))
    \end{equation}
    with a function $\kappa: \mathbb{R}^d \to \mathbb{R}_+$ as in Proposition \ref{prop:existence_kappa}. Then there exist positive real numbers $\tilde{\delta}_0$ and $\tilde{T}_0$ such that when $\delta \in(0, \tilde{\delta}_0]$ and $T \geq \tilde{T}_0$, we have
    \begin{equation}
        \mathcal{W}_{1,\Tilde{\rho}}\left(\mu \Tilde{Q},\nu \Tilde{Q}\right) \leq \left(1-\left(h - \tilde{h}\right)\delta\right) \mathcal{W}_{1,\Tilde{\rho}}\left(\mu ,\nu \right), \quad \forall \mu,\nu \in \mathcal{P}(\mathbb{R}^{dN}) \,,
    \end{equation}
    where $h$ is given by \eqref{eq:h} and $\tilde{h}$ is defined as
    \begin{equation}\label{eq: defhTilde}
        \tilde{h}:=(\delta L_G^2 +2 C_G + 2\delta L_F C_G)(\tilde{T}_0+2 \| \kappa \|_{\infty}) \,. 
    \end{equation}
    The explicit definition of $\tilde{\delta}_0$ and $\tilde{T}_0$ is given in \eqref{eq: tilde_delta_0} and \eqref{eq: tildeT_0}.
\end{theorem}
Note that Theorem \ref{thm:rhoContractionForSystem} only gives us an $L^2$ contraction when $h > \tilde{h}$. Comparing the formulas for $h$ and $\tilde{h}$, it is evident that $h > \tilde{h}$ holds only when the influence of the interaction component $G$ of the drift is weaker than the influence of the confinement component $F$. This is consistent with other existing ergodicity results for McKean-Vlasov SDEs in the literature. 

\begin{corollary}\label{cor:particles}
    Suppose Assumptions \ref{hypo: F} and \ref{hypo: G} hold. When $\delta \in(0,\tilde{\delta}_0]$ and $T \geq \tilde{T}_0$, there exists $M>1$ such that the transition kernel $\Tilde{Q}$ satisfies, for any $k \geq 1$, 
    \begin{equation}
        \mathcal{W}_2\left(\mu \Tilde{Q}^k,\nu \Tilde{Q}^k\right) \leq M\left(1-\left(h - \tilde{h}\right)\delta\right)^k \mathcal{W}_2\left(\mu ,\nu \right), \quad \text{ for all } \mu,\nu \in \mathcal{P}(\mathbb{R}^{dN}).
    \end{equation}
\end{corollary}
Corollary \ref{cor:particles} follows from Theorem \ref{thm:rhoContractionForSystem} in the same way as Corollary \ref{cor: WContraction} follows from Theorem \ref{thm:rhoContraction}. The formula for $M$ depends only on $T$ and $\kappa$, and is hence independent of the step size $\delta$ and the number of particles $N$.

\subsection{Euler scheme, weak \texorpdfstring{$L^2$}{L^2}-Wasserstein contraction and local Poincar\'{e} inequality}
We say that  a Markov transition kernel $\Bar{Q}$ on $\mathbb{R}^d$  satisfies the local Poincar\'{e} inequality with constant $C_{LP}$ if for all $x\in \mathbb{R}^d$ the measure $\Bar{Q}(x,\cdot)$ satisfies the Poincar\'{e} inequality with constant $C_{LP}$, namely, we have for any $f \in \mathcal{C}^\infty(\mathbb{R}^d)$ with polynomial growth
\begin{equation}\label{eq: localPoincare}
    \Bar{Q}f^2(x) - (\Bar{Q}f)^2(x) \leq C_{LP}\Bar{Q}\lvert \nabla f\rvert^2(x), \quad \text{ for all } x\in \mathbb{R}^d.
\end{equation}
We demonstrate that weak $L^2$-Wasserstein contraction and local Poincar\'{e} inequality of $\bar{Q}$ are sufficient conditions to establish the local Poincar\'{e} inequality for $\Bar{Q}^k$ with any $k\in \mathbb{N}$, as well as for the invariant measure $\pi_\infty$ of the kernel $\Bar{Q}$.
\begin{theorem}\label{thm:PoincareInequality}
Let $\bar{Q}$ be a Markov transition operator on $\mathbb{R}^d$. Suppose $\bar{Q}$ satisfies the weak $L^2$-Wasserstein contraction in the sense that there exist $\Bar{M}>1>\Bar{w}>0$ such that
\begin{equation}\label{eq: W2ContractionGeneral}
    \mathcal{W}_2(\mu \bar{Q}^k, \nu \bar{Q}^k)\leq \bar{M} (1-\bar{w})^k \mathcal{W}_2(\mu , \nu ), \quad \text{ for all } \mu,\nu \in \mathcal{P}(\mathbb{R}^d),\  \text{ and for all } k \in \mathbb{N}.
\end{equation}
Suppose moreover that $\Bar{Q}$ satisfies the local Poincar\'{e} inequality with constant $C_{LP}$. Then for any $k \in \mathbb{N}$, $\bar{Q}^k$ satisfies the local Poincar\'{e} inequality with constant $\frac{C_{LP} \bar{M}^2}{1-(1-\bar{w})^2}$. Moreover, the invariant measure of $\Bar{Q}$ also satisfies the Poincar\'{e} inequality with the same constant.
\end{theorem}

\begin{remark}
This result (and its proof) is similar to \cite[Theorem 1]{MalrieuTalay} and \cite[Lemma 3.2]{MonmarchePDMP}. The difference with \cite{MalrieuTalay} is that the Poincar\'{e} constant in our Theorem~\ref{thm:PoincareInequality} is uniform in time (and hence holds for the invariant measure), and the difference with \cite{MonmarchePDMP} is that Theorem~\ref{thm:PoincareInequality} allows $\bar M>1$, i.e., $\bar Q^k$ only has to be a $\mathcal W_2$-contraction for $k$ large enough.
\end{remark}

Note that the result above is formulated for an arbitrary transition kernel $\bar{Q}$ satisfying the local Poincar\'{e} inequality and the Wasserstein contraction condition \eqref{eq: W2ContractionGeneral}. Since for the transition kernel $Q$ of the Euler scheme \eqref{eq:chain} we always have the local Poincar\'{e} inequality (see Section \ref{sec: proof_local_Poincare} for details), and under Assumptions \ref{hypo: F} and \ref{hypo: G} we have \eqref{eq: W2ContractionGeneral} (for $T$ large enough), we  obtain the following corollary.

\begin{corollary}\label{cor:Poincare}
    Let $Q$ be the transition kernel of the Euler scheme \eqref{eq:chain}. Let Assumptions~\ref{hypo:1} and \ref{hypo:2} hold, and let $h,\delta_0,T_0,M$ be as in Corollary~\ref{cor: WContraction}. Assume that $\delta<\delta_0$ and $T\geqslant T_0$. Then, for all $k \geq 1$, $Q^k$ satisfies the local Poincar\'{e} inequality with   constant $\frac{2TM^2}{2h - h^2\delta}$. Moreover, the invariant measure $\pi_{\infty}$ of $Q$ satisfies the Poincar\'{e} inequality with the same constant.
\end{corollary}

Remarkably, the Poincar\'{e} constant in this result is uniformly bounded over small values of the step size $\delta$.

The proofs of Theorem \ref{thm:PoincareInequality} and Corollary \ref{cor:Poincare} will be presented in Section \ref{sec: proof_local_Poincare}.

\section{Proofs of Theorem \ref{thm:rhoContraction} and Corollary \ref{cor: WContraction}}\label{sec: proofForOneParticle}
Let us begin by outlining the ranges of values for $\delta$ and $T$ that are sufficient for proving Theorem \ref{thm:rhoContraction} and Corollary \ref{cor: WContraction}. In the entirety of this section, Assumptions~\ref{hypo:1} and \ref{hypo:2} are enforced and we will be working with a function $\kappa : \mathbb{R}^d \to \mathbb{R}$ whose existence is guaranteed by Proposition~\ref{prop:existence_kappa}. 
In particular, we fix $a\in (12K,\infty)$ and $L \in (0,c/6)$. 
The range of $T$ is defined as follows. 
\begin{equation}\label{eq: T_0}
        T \geq T_0 := \max \{T_1,T_2,T_3\}
    \end{equation}
    where
    \begin{align*}\label{eq: T123}
	&T_1 := \frac{2 \lVert\nabla \kappa \rVert_\infty  \sup_{\lvert y \rvert \leq R} \lvert b(y)\rvert}{a}\\
	&T_2:= \frac{2 \lVert\nabla \kappa \rVert_\infty \sup_{\lvert y \rvert \leq \bar{R}} \lvert b(y)\rvert}{L}\\
	&T_3:= 2 \lVert \kappa \rVert_\infty
    \end{align*}
    with
    \begin{equation}
        \bar{R}:=\frac{R_*+\delta \lvert b(0)\rvert}{1-\delta L_b},
    \end{equation}
    where we recall that $R_*$ is defined in Proposition~\ref{prop:existence_kappa}. Bounds for $\lVert \kappa \rVert_\infty$ and  $\lVert \nabla \kappa \rVert_\infty$ are provided in \eqref{eq:bound_kappa_with_epsilon} and \eqref{eq:bound_kappa}. The range of $\delta$ is defined as follows.
\begin{equation}\label{eq: delta0}
0<\delta \leq \delta_0:= \min (\delta_1, \delta_2,\delta_3,\delta_4)
\end{equation}
where
\begin{equation}\label{eq:delta1234}
    \delta_1:= \frac{1}{L_b} \quad \delta_2:=\frac{c}{L_b^2} \quad \delta_3:= \frac{K}{L_b^2} \quad \delta_4 := \frac{R^2 d \Gamma(\frac{d}{2})}{8T(d+2) \left(d \Gamma(\frac{d}{2}) +2T\Gamma(\frac{d+2}{2})\right)}\,.
\end{equation}
In order to establish the validity of Theorem \ref{thm:rhoContraction}, we rely on the following proposition.

\begin{proposition}\label{prop: properties_kappa_b}
Let $Z$ be a $d$-dimensional standard normal random variable. Suppose Assumptions \ref{hypo:1} and \ref{hypo:2} hold and let $\kappa$, $a$ and $L$ be as in Proposition \ref{prop:existence_kappa}. Suppose further $T\geq\max\{T_1, T_2\}$ and $0<\delta \leq \delta_1$. Then, the following holds.\\
i) For all $x\in \mathbb{R}^d$ with $\lvert x \rvert \leq R$, we have
\begin{equation}\label{eq:bound_a}
   \mathbb{E}\left[\kappa \left(x+ \delta b(x) + \sqrt{2 \delta T}Z\right)\right] \leq \kappa(x) -\frac{a}{2} \delta T.
\end{equation}
ii) For all $x\in \mathbb{R}^d$, we have
\begin{equation}\label{eq:bound_L}
    \mathbb{E}\left[\kappa \left(x+ \delta b(x) + \sqrt{2 \delta T}Z\right)\right] \leq \kappa(x) + \frac{3}{2}L \delta T.
\end{equation}
\end{proposition}

\begin{proof}
Let $x\in \mathbb{R}^d$ with $\lvert x \rvert \leq R$. Due to the mean-value theorem and condition \eqref{eq:condition_bound_a}, we have
\begin{align*}
    \mathbb{E}\left[\kappa \left(x+ \delta b(x) + \sqrt{2 \delta T}Z\right) \right] &\leq \mathbb{E}\left[\kappa \left(x+ \sqrt{2\delta T}Z \right)\right] + \delta \lVert \nabla \kappa \rVert_\infty \sup_{\lvert y \rvert \leq R} \lvert b(y)\rvert  \\
    & \leq \kappa(x) - a \delta T + \delta\lVert \nabla \kappa \rVert_\infty \sup_{\lvert y \rvert \leq R} \lvert b(y)\rvert \,.
\end{align*}
Since $T \geq T_1$, we obtain \eqref{eq:bound_a}. Now that the first item is proved, we take $x\in \mathbb{R}^d$ such that $\lvert x \rvert \geq R$ and we consider two cases.
\begin{itemize}
    \item For the case $|x|\geqslant \bar R$, we first observe that by the triangle inequality and the Lipschitz condition of $b$, we have
    \begin{equation*}
	\lvert x + \delta b(x) \rvert \geq  \lvert x \rvert  - \delta \lvert b(x)\rvert 
	 \geq \lvert x \rvert - \delta L_b \lvert x \rvert -\delta \lvert b(0)\rvert
	 =(1-\delta L_b) \lvert x \rvert - \delta \lvert b(0)\rvert.
\end{equation*}
Hence, when $\delta < \delta_1$ we have
\begin{equation}\label{eq:greatthanRstar}
  \lvert x \rvert \geq \bar{R}  \Rightarrow \lvert x + \delta b(x)   \rvert \geq R_*.
\end{equation}
Since $\kappa(y)=0$ for $|y|\geqslant R_*$, for $|x|\geqslant \bar R$, we simply apply \eqref{eq:condition_bound_L} (with $x$ replaced by $x+\delta b(x)$) to get
\[    \mathbb{E}\left[\kappa (x+\delta b(x) + \sqrt{2 \delta T}Z)\right] \leq \kappa(x + \delta b(x)) + L \delta T  = L \delta T = \kappa(x) +L \delta T,\]
where in the last step we used the fact that $R_*<\bar{R}$.
 \item
 For the case $ |x|< \bar R$, 
we use \eqref{eq:condition_bound_L} to obtain
\begin{align*}
       \mathbb{E}\left[\kappa \left(x+ \delta b(x) + \sqrt{2 \delta T}Z\right)\right] &\leq \mathbb{E}\left[\kappa\left(x+ \sqrt{2\delta T}Z\right)\right] + \delta \lVert \nabla \kappa \rVert_\infty \sup_{\lvert y \rvert \leq \bar{R}}\lvert b(y)\rvert\\
       & \leq \kappa(x) + L \delta T + \delta \lVert \nabla \kappa \rVert_\infty \sup_{\lvert y \rvert \leq \bar{R}}\lvert b(y)\rvert.
\end{align*}
Since we assume $T \geq T_2$, we obtain \eqref{eq:bound_L}, which concludes the proof of the proposition. 
\end{itemize}
\end{proof}
We will now apply Proposition \ref{prop: properties_kappa_b} to prove Theorem \ref{thm:rhoContraction}.
\begin{proof}[Proof of Theorem \ref{thm:rhoContraction}]
Let $x$ and $y$ be any two points in $\mathbb{R}^d$. We consider two Markov chains $(X_k)_{k\geq 0}$ and $(Y_k)_{k\geq 0}$ that start at $x$ and $y$ respectively and use the same Gaussian variables at each step. In other words, these chains are defined as follows: for $k \in \mathbb{N}$,
\begin{align*}
    X_{k+1} &= X_k + \delta b(X_k) + \sqrt{2\delta T} Z_k \\
    Y_{k+1} &= Y_k + \delta b(Y_k) + \sqrt{2\delta T} Z_k .
\end{align*}
 By the definition of $X_1$ and $Y_1$, we have
\begin{multline*}
    \mathbb{E} \left[\rho(X_1,Y_1) \right] = \mathbb{E} \left[\lvert x+\delta b(x) -y - \delta b(y)\rvert^2 (T +\kappa (X_1)+ \kappa(Y_1))\right] \\
     = \left( \lvert x- y\rvert^2 +2(x-y) \cdot \delta(b(x)-b(y))+ \delta^2 \lvert b(x) -b(y) \rvert^2 \right) \mathbb{E}\left[T +\kappa (X_1)+ \kappa(Y_1)\right].
\end{multline*}
\begin{itemize}
    \item 
When $\lvert x \rvert \geq R$ or $\lvert y \rvert \geq R$, by using first Assumption \ref{hypo:1} and \ref{hypo:2} and then \eqref{eq:bound_L} we have
\begin{eqnarray}
    \mathbb{E}[\rho(X_1,Y_1)] & \leq &  (1-2c \delta +L_b^2 \delta^2)  \left (T +\mathbb{E}[\kappa(X_1)] + \mathbb{E}[\kappa(Y_1)] \right)\lvert x-y \rvert^2 \nonumber\\
    &\leq& (1-2c \delta +L_b^2 \delta^2)  (T +\kappa(x) + \kappa(y) +3L \delta T )\lvert x-y \rvert^2 \nonumber\\
    &=& (1-2c \delta +L_b^2 \delta^2)\left(1+ \frac{3L\delta T}{T + \kappa(x) +\kappa(y)}\right) \rho(x,y)\nonumber\\
    & \leq & (1-2c \delta +L_b^2 \delta^2)\left(1+ 3L \delta \right)\rho(x,y)\,,\label{eq:to_modify_L}
\end{eqnarray}
where we used that  $\kappa$ is positive.
\item
If $\lvert x \rvert \leq R$ and $\lvert y \rvert \leq R$, using similarly Assumptions \ref{hypo:1} and \ref{hypo:2} and then \eqref{eq:bound_a} yields
\begin{eqnarray}
    \mathbb{E}[\rho(X_1,Y_1)] &\leq& (1+2 K \delta + L_b^2\delta^2)  \left (T +\mathbb{E}[\kappa(X_1) + \mathbb{E}[\kappa(Y_1)] \right)\lvert x-y \rvert^2\nonumber\\
    &\leq &(1+2K\delta + L_b^2\delta^2)  ( T +\kappa(x) +\kappa(y) -a\delta T)\lvert x-y \rvert^2\nonumber\\
    &=& (1+2 K \delta +  L_b^2\delta^2) \left(1- \frac{a\delta T}{T +  \kappa(x) + \kappa(y)} \right)\rho(x,y)\nonumber \\
    & \leq & (1+2 K \delta+  L_b^2 \delta^2) \left (1- \frac{1}{2} a\delta \right)\rho(x,y)\,,\label{eq:to_modify_a}
\end{eqnarray}
where we used in the last line that  $\kappa$ is bounded and $T \geq T_3$.
\end{itemize}
In view of \eqref{eq:to_modify_L} and \eqref{eq:to_modify_a}, it remains to determine $h>0$ such that, for all $\delta $ in the considered range,
\begin{equation*} 
    \max\left\{(1-2c \delta +L_b^2 \delta^2)\left(1+ 3L \delta \right) \ , \ (1+2 K \delta +  L_b^2 \delta^2 ) \left (1- \frac{1}{2} a\delta \right)\right \} < 1 - h \delta\,.
\end{equation*}
Thanks to $\delta \leq c/L_b^2$ and $L \leq c/6$ (cf.\ Proposition \ref{prop:existence_kappa}), we have
\begin{equation*}
    (1-2c \delta +L_b^2 \delta^2)\left(1+ 3L \delta \right) \leqslant (1-c \delta )\left(1+ c \delta/2 \right)
\end{equation*}
and thanks to $\delta \leq K/L_b^2$ and $a \geq 12K$, we have
\begin{equation*}
    (1+2 K \delta +  L_b^2 \delta^2 ) \left (1- \frac{1}{2} a\delta \right) \leqslant \left(1+\frac14\delta a\right )\left(1-\frac{1}{2} a \delta\right).
\end{equation*}
Thus we only to have to bound $(1-x)(1+x/2)$ for $x \geq 0$, and we simply see that
\begin{equation*}
(1-x)(1+x/2) = 1 - x/2 - x^2/2 \leqslant 1-x/2
\end{equation*}
which implies that we have for all $x,y \in \mathbb{R}^d$
\begin{equation*}
    \mathbb{E}_{x,y}[\rho(X_1,Y_1)] \leq (1-h\delta)\rho(x,y),
\end{equation*}
where we could choose 
\begin{equation}\label{eq:h}
    h =\min(c/2,a/4).
\end{equation}
By a standard argument (see, for example, \cite[Lemma 2.1]{EberleMajka2019}), this gives
\begin{equation*}
    \mathcal{W}_{1,\rho} \left(\mu Q, \nu Q\right) \leq (1-h\delta ) \mathcal{W}_{1,\rho}(\mu , \nu ) \quad \text{ for all } \mu,\nu \in \mathcal{P}(\mathbb{R}^d).
\end{equation*}
\end{proof}
\begin{proof}[Proof of Corollary \ref{cor: WContraction}]
    Let $\mu$ and $\nu$ be any two measures in $\mathcal{P}(\mathbb{R}^d)$. 
    For a $k\in \mathbb{N}$, using first that $\kappa \geq 0$, then Theorem \ref{thm:rhoContraction} and finally that $\kappa$ is bounded, we get 
    \begin{align*}
        T \mathcal{W}_2 \left(\mu Q^k, \nu Q^k\right) &\leq \mathcal{W}_{1,\rho} \left(\mu Q^k, \nu Q^k \right)\\
        &\leq (1-h \delta )^k \mathcal{W}_{1,\rho}(\mu,\nu)\\
        & \leq (1-h \delta )^k(T + 2 \lVert \kappa \rVert_\infty) \mathcal{W}_2(\mu,\nu).
    \end{align*}
     Dividing both sides by $T$, 
     we conclude that the transition kernel $Q$ induces a weak $L^2$-Wasserstein contraction with the rate of convergence contraction $h \delta$, and
    \begin{equation}\label{eq: explicit_M}
    M:= \left(1+\frac{2\lVert \kappa \rVert_\infty}{T} \right).
    \end{equation}
\end{proof}

\section{Proof of Proposition \ref{prop:existence_kappa}}\label{sec:existenceKappa}
In this section, we aim to construct a function $\kappa$ that satisfies the conditions specified in Proposition \ref{prop:existence_kappa}. 
To this end, we first prove two auxiliary results.
\begin{proposition}\label{prop:eta_1}
Let $\alpha_1, \beta_1,R,T \in \mathbb{R}_+$. Suppose $\eta_1 : \mathbb{R}^d \to \mathbb{R}$ satisfies the following two conditions:
\begin{itemize}
    \item $\eta_1 (x) = \alpha_1 -\beta_1 \lvert x \rvert^2 , \quad \text{ for all } x \in \mathbb{R}^d \text{ with } \lvert x \rvert \leq 2R$,\\
    \item $\eta_1$ is bounded with $\lVert \eta_1\rVert_\infty = \alpha_1$.
\end{itemize}
If $\delta \in (0, \delta_4)$, with $\delta_4$ defined in \eqref{eq: delta0}, then for all $x \in \mathbb{R}^d$ with $\lvert x \rvert \leq R$, 
\begin{equation}\label{eq:inequality_with_beta1}
     \quad \mathbb{E}\left[\eta_1\left(x+ \sqrt{2 \delta T}Z\right)\right] \leq \eta_1(x) - \beta_1 d \delta T,
\end{equation}
where $Z$ is a random vector which has the standard normal distribution on $\mathbb{R}^d$. In particular, we have \eqref{eq:condition_bound_a} (for $\kappa=\eta_1$) by taking $\beta_1 = \frac{a}{d}$.
\end{proposition}
\begin{proof}
Take $x \in \mathbb{R}^d$ such that $\lvert x \rvert \leq R$. By the definition of $\eta_1$, we have 
\begin{align*}
    \mathbb{E}\left[\eta_1\left(x + \sqrt{2\delta T} Z\right)\right] &= \mathbb{E} \left[ \alpha_1 - \beta_1 \left\lvert x + \sqrt{2\delta T} Z \right\rvert^2 \right]\\
    &\underbrace{+\mathbb{E}\left[\left(\eta_1\left(x + \sqrt{2\delta T} Z\right) -\alpha_1 + \beta_1 \left\lvert x + \sqrt{2\delta T} Z \right\rvert^2  \right) \mathds{1}_{\left\lvert x + \sqrt{2\delta T} Z \right \rvert \geq 2R} \right]}_{:=(*)}\\
    &= \alpha_1 - \beta_1 \mathbb{E} \left[ \lvert x \rvert^2  + 2 \sqrt{2 \delta T} x \cdot Z + 2 \delta T \lvert Z \rvert^2 \right] + (*)\\
    &=  \eta(x) -  2\beta_1 d \delta T +(*)\,.
\end{align*}
We will now demonstrate that $(*) \leqslant \beta_1 d \delta T$ for all $\delta \in (0,\delta_4)$. Since $\lVert \eta_1 \rVert_\infty = \alpha_1$,
\begin{align}
     (*) 
    & \leq   \beta_1  \mathbb{E} \left[\left\lvert x + \sqrt{2\delta T} Z \right\rvert^2    \mathds{1}_{\left \lvert x + \sqrt{2\delta T} Z \right \rvert \geq 2R} \right]\nonumber\\
    &\leq 2  \beta_1  \mathbb{E} \left[ \left(  \lvert x \rvert^2 +2  \delta T \lvert Z \rvert^2   \right)\mathds{1}_{\left  \lvert x + \sqrt{2\delta T} Z \right \rvert \geq 2R} \right]\nonumber\\
    & \leq 2 \beta_1 R^2  \underbrace{\mathbb{P} \left(\lvert Z \rvert \geq \frac{R }{\sqrt{2\delta T}} \right)}_{(\star)}+ 4 \beta_1 \delta T \underbrace{\mathbb{E} \left[ \lvert Z \rvert^2\mathds{1}_{\lvert Z \rvert \geq \frac{R }{\sqrt{2\delta T}}} \right]}_{(\star \star)}\,,\label{eq:minous_add_star}
\end{align}
where we used that $|x|\leqslant R$ and that $\left\lvert x + \sqrt{2\delta T} Z \right\rvert \geq 2R$ implies $\lvert Z \rvert \geq \frac{R}{\sqrt{2\delta T}}$. 
By using the Markov Inequality, we get
\begin{equation}\label{eq:gammabound1}
(\star) \leq \frac{4 \delta^2 T^2 \mathbb{E}[\lvert Z \rvert^4]}{R^4}=\frac{4\delta^2T^2d(d+2)}{R^4}.
\end{equation}
We apply Markov's inequality with the fourth moment of $Z$, because it allows us to obtain a squared $\delta$ term on the right-hand side, which will be useful in subsequent steps. For the term $(\star \star)$, we note that $\lvert Z \rvert^2$ follows the chi-square distribution with $d$ degrees of freedom. Therefore, we can write
\[
    (\star \star) = \int_{\frac{R^2}{2\delta T}}^\infty  x \frac{1}{2^{\frac{d}{2}} \Gamma(\frac{d}{2})} x^{\frac{d}{2}-1} e^{-\frac{x}{2}}dx\\
    =\frac{2\Gamma(\frac{d+2}{2})}{\Gamma(\frac{d}{2})} \mathbb{P}\left(U \geq \frac{R^2}{2\delta T}\right),
\]
where $U$ has the chi-square distribution with $d+2$ degrees of freedom. Applying the Markov Inequality and using $\mathbb{E}[|U|] = d+2$, we have 
\begin{equation}\label{eq:gammabound2}
    (\star \star) \leq \frac{4 (d+2) \delta T \Gamma(\frac{d+2}{2})}{R^2\Gamma(\frac{d}{2})}.
\end{equation}
Using \eqref{eq:gammabound1} and  \eqref{eq:gammabound2} in \eqref{eq:minous_add_star} gives
\[
   (*) \leq \frac{8\beta_1 \delta^2 T^2 d(d+2)}{R^2} +\frac{16\beta_1\delta^2 T^2 (d+2) \Gamma(\frac{d+2}{2})}{R^2 \Gamma(\frac{d}{2})}\,,\]
   which is less than $\beta_1 d\delta T$ when $\delta \leqslant \delta_4$, which concludes the proof. 
\end{proof}

\begin{proposition}\label{prop:eta_2}
Suppose $\eta_2 : \mathbb{R}^d \to \mathbb{R}$ is a function for which there exists $\beta_2 >0$ such that $f(x):= \eta_2(x) -\beta_2 \lvert x \rvert^2$ is concave. Then, for any $T,\delta>0$, considering $Z$ a standard $d$-dimensional Gaussian variable, $\eta_2$ satisfies 
\begin{equation*}
    \mathbb{E}\left[\eta_2\left(x+\sqrt{2 \delta T}Z\right)\right] \leq \eta_2(x) + 2 \beta_2 d \delta T \text{ for all } x \in \mathbb{R}^d. 
\end{equation*}
In particular, we have \eqref{eq:condition_bound_L} (for $\kappa=\eta_2$) by taking $\beta_2 = \frac{L}{2d}$.
\end{proposition}

\begin{proof}
By Jensen's inequality, we have
\begin{equation*}
    \mathbb{E}\left[\eta_2\left(x + \sqrt{2\delta T}Z\right) - \beta_2 \left\lvert x + \sqrt{2\delta T}Z \right \rvert^2 \right] \leq \eta_2(x) - \beta_2 \lvert x \rvert^2.
\end{equation*}
Hence we have 
\begin{align*}
     \mathbb{E}\left[\eta_2\left(x + \sqrt{2\delta T}Z\right)\right] &\leq \eta_2(x) - \beta_2 \lvert x \rvert^2 + 
     \mathbb{E}\left[\beta_2 \left\lvert x+ \sqrt{2\delta T}Z \right\rvert^2 \right]\\
     &= \eta_2(x) + 2 \beta_2 d \delta T .
\end{align*}
\end{proof}

\begin{proof}[Proof of Proposition \ref{prop:existence_kappa}]
We need to construct a function $\kappa$ that satisfies items (1)-(4). Based on Propositions \ref{prop:eta_1} and \ref{prop:eta_2}, we see that it is sufficient to verify the following conditions (1')-(5'):
\begin{enumerate}
    \item[(1')] $\kappa$ is $\mathcal{C}^1(\mathbb{R}^d)$.
    \item[(2')] There exists $R_*>R$ such that 
    \begin{equation*}
        \kappa(x) = 0 \text{ and } \nabla \kappa(x) =0, \quad \forall x \in \mathbb{R}^d \text{ with }\lvert x \rvert \geq R_*
    \end{equation*}
    \item[(3')] $\kappa(x) = \alpha_1-\frac{a}{d} \lvert x \rvert^2$ for $x$ such that $\lvert x \rvert \leq 2R$.
    \item[(4')] $\lVert \kappa \rVert_\infty = \alpha_1$ and $\kappa\geq0$.
    \item[(5')]  $\kappa(x) - \frac{L}{2d}\lvert x \rvert^2$ is concave.
\end{enumerate}
To this end, we define $\kappa$ as follows, taking $0<\epsilon< \frac{L}{2d} $
\begin{equation}
    \kappa(x):= \begin{cases}
        \alpha_1 - \frac{a}{d} \lvert x \rvert^2 \quad &\text{if } 0\leq \lvert x \rvert \leq 2R\\
        \left(\frac{L}{2d}-\epsilon\right) (\lvert x \rvert - \alpha_2)^2 \quad &\text{if } 2R< \lvert x \rvert \leq R_*\\
        0 \quad &\text{if } \lvert x \rvert > R_*
    \end{cases}
\end{equation}
and we choose $\alpha_1$, $\alpha_2$, and $R_*$ so that $\kappa$ belongs to $\mathcal{C}^1(\mathbb{R}^d)$. The explicit values of $\alpha_1$, $\alpha_2$ and $R_*$ can be found by solving the following equations
\begin{equation}
    \begin{cases}
        \left. \nabla \left( \alpha_1 - \frac{a}{d} \lvert x\rvert^2 \right)\right|_{x\in \mathbb{R}^d, \lvert x \rvert =2R}= \left.\nabla \left ( \left(\frac{L}{2d}-\epsilon \right) (\lvert x \rvert-\alpha_2)^2\right)\right|_{x\in \mathbb{R}^d, \lvert x \rvert =2R}\\
        \left. \left( \alpha_1 - \frac{a}{d} \lvert x\rvert^2 \right)\right|_{x\in \mathbb{R}^d, \lvert x \rvert =2R}= \left.\left ( \left(\frac{L}{2d}-\epsilon \right) (\lvert x \rvert-\alpha_2)^2\right)\right|_{x\in \mathbb{R}^d, \lvert x \rvert =2R}\\
        \left.\nabla \left ( \left(\frac{L}{2d}-\epsilon \right) (\lvert x \rvert-\alpha_2)^2\right)\right|_{x\in \mathbb{R}^d, \lvert x \rvert =R_*} = 0
    \end{cases}.
\end{equation}
We get
\begin{equation}
    \alpha_1 = 4aR^2\left(\frac{1}{d}+\left(\frac{L}{2d}-\epsilon \right)\frac{4a}{(L-2d\epsilon)^2}\right)  \quad \text{and} \quad  R_* = \alpha_2 = 2R\left(1+\frac{2a}{L-2d\epsilon}\right).
\end{equation}
Hence the explicit definition is 
\begin{equation}\label{eq: kappa}
    \kappa(x) = \begin{cases}4aR^2\left(\frac{1}{d}+\left(\frac{L}{2d}-\epsilon \right)\frac{4a}{(L-2d\epsilon)^2}\right) - \frac{a}{d} \lvert x \rvert^2 \quad &\text{if } 0\leq \lvert x \rvert \leq 2R\\
    \left( \frac{L}{2d}-\epsilon\right) \left(\lvert x \rvert - 2R\left(1+\frac{2a}{L-2d\epsilon}\right) \right)^2 \quad &\text{if } 2R< \lvert x \rvert \leq R_*\\
    0 \quad &\text{if } \lvert x \rvert > R_*
    \end{cases}
\end{equation}
where $ R_* = 2R\left(1+\frac{2a}{L-2d\epsilon}\right)$. We also have 
\begin{equation}\label{eq:bound_kappa_with_epsilon}
\lVert \kappa \lVert_\infty = 4aR^2\left(\frac{1}{d}+\left(\frac{L}{2d}-\epsilon \right)\frac{4a}{(L-2d\epsilon)^2}\right) \quad \text{and} \quad \lVert \nabla \kappa \rVert_\infty = \frac{4 a R}{d}.
\end{equation}
We can now confirm that the function $\kappa$, defined in equation \eqref{eq: kappa}, satisfies (1')-(4') by definition. Additionally, we will demonstrate that it also satisfies (5'). Consider the following function $g:\mathbb{R}_+ \to \mathbb{R}$ defined as 
\begin{equation}
    g(y) = \begin{cases}4aR^2\left(\frac{1}{d}+\left(\frac{L}{2d}-\epsilon \right)\frac{4a}{(L-2d\epsilon)^2}\right) - \frac{a}{d} y^2 -\frac{L}{2d}y^2\quad & y \in [0,2R]\\
    \left( \frac{L}{2d}-\epsilon\right) \left(y - 2R\left(1+\frac{2a}{L-2d\epsilon}\right) \right)^2 -\frac{L}{2d}y^2\quad & y \in [2R,R_*]\\
    -\frac{L}{2d}y^2 \quad & y \in [R_*,\infty)
    \end{cases}
\end{equation}
and function $h: \mathbb{R}^d \to \mathbb{R}_+$ defined as $h(x) = \lvert x \rvert$. From the equation $\kappa(x) - \frac{L}{2d}\lvert x \rvert^2 = g(h(x))$, we can observe that if $g$ is both non-increasing and concave, then $\kappa(x) - \frac{L}{2d}\lvert x \rvert^2 $ is also concave, as shown in \cite[Section 3.2.4]{boyd2004convex}. To verify that $g$ satisfies these conditions, we differentiate it as follows
\begin{equation}
    g'(y) = \begin{cases} -2\left( \frac{a}{d} + \frac{L}{2d}\right) y \quad & y \in [0,2R]\\
    - 2\epsilon y - 2 \left( \frac{L}{2d} - \epsilon \right) 2R (1+\frac{2a}{L-2d\epsilon})\quad & y \in [2R,R_*]\\
    -\frac{L}{d} y \quad & y \in [R_*,\infty).
    \end{cases}
\end{equation}
Since $g'(x)$ is negative for all $x$ by assumption, $g$ is a decreasing function. Moreover, the fact that $g'$ decreases implies that $g$ is concave. Therefore, we have shown that the function $\kappa$ defined in equation \eqref{eq: kappa} satisfies all the required properties (1')-(5').
\end{proof}

\begin{remark}
For the specific choice of $\kappa$ given by \eqref{eq: kappa}, we can select
\begin{equation*}
    a= 12K, \quad L =\frac{c}{6}  \quad \text{and } \epsilon = \frac{c}{42d}
\end{equation*}
and obtain
\begin{equation}\label{eq:bound_kappa}
\lVert \kappa \lVert_\infty = \left(1+\frac{84K}{c}\right) \frac{48 K R^2}{d} \quad \text{and} \quad \lVert \nabla \kappa \rVert_\infty = \frac{48 K R}{d}.
\end{equation}
Then, plugging \eqref{eq:bound_kappa} into the formulas for $T_0$, $\delta_0$, and $h$ obtained in Theorem \ref{thm:rhoContraction} and Corollary \ref{cor: WContraction}, we get explicit bounds on the ranges of parameters $T$ and $\delta$, and an explicit contraction rate $h$. 
\end{remark}

\section{Proofs of Theorem \ref{thm:rhoContractionForSystem}}\label{sec:proofInteracting}
Similarly as in Section \ref{sec: proofForOneParticle}, we begin by specifying the ranges of $\delta$ and $T$. The lower bound on $T$ is given by 
\begin{equation}\label{eq: tildeT_0}
    \tilde{T}_0:=\max(\tilde{T}_1, \tilde{T}_2,\tilde{T}_3),
\end{equation}
where
\begin{align*}
    &\tilde{T}_1:= \frac{2 \left(\sup_{\lvert y^{(i)}\rvert \leq R} \lvert F(y^{(i)})\rvert + M_G+ M_G R^p \right)\lVert \nabla \kappa \rVert_\infty}{a}\\
    &\tilde{T}_2:=\frac{ 2 \left(\sup_{\lvert y^{(i)}\rvert \leq \tilde{R}} \lvert F(y^{(i)})\rvert +M_G + M_G \tilde{R}^p \right) \lVert \nabla \kappa \rVert_\infty}{L}\\
    &\tilde{T}_3:= 2\lVert \kappa \rVert_\infty.
\end{align*}
with
$$\tilde{R}:= \max \left(1, \frac{R_*+\delta F(0) +\delta M_G}{1-\delta L_F -\delta M_G} \right).$$
The upper bound on $\delta$ is given by
\begin{equation}\label{eq: tilde_delta_0}
\tilde{\delta}_0:=\min(\tilde{\delta}_1,\tilde{\delta}_2,\tilde{\delta}_3,\tilde{\delta}_4),
\end{equation}
where
\begin{equation*}
    \tilde{\delta}_1 := \frac{1}{L_F+M_G} \quad
    \tilde{\delta_2}:= \frac{c}{L_F^2} \quad
    \tilde{\delta}_3:= \frac{K}{L_F^2} \quad
    \tilde{\delta}_4 = \frac{R^2 \Gamma(\frac{d}{2})}{T(d+2)\Gamma(\frac{d}{2}) +16T\Gamma(\frac{d+2}{2})}.
\end{equation*}
Similarly to Proposition \ref{prop: properties_kappa_b}, the weight function $\kappa$ has the following properties for appropriate ranges of $\delta$ and $T$.
\begin{proposition}\label{prop:kappa_system}
Let $Z$ be a $d$-dimensional standard normal random variable. Suppose Assumptions \ref{hypo: F} and \ref{hypo: G} hold and let $\kappa$ be as in Proposition \ref{prop:existence_kappa}. Suppose furthermore $T\geq \max\{\Tilde{T}_1, \Tilde{T}_2\}$ and $\delta \in (0,\Tilde{\delta}_1]$. Then $\kappa$ satisfies the following two items.\\
i) For all $i \in \llbracket 1,N\rrbracket$ and for all $\textbf{x}\in \mathbb{R}^{dN}$ such that $\lvert x^{(i)} \rvert \leq R$, we have
\begin{equation}\label{eq:bound_a_system}
   \mathbb{E}\left[\kappa \left(x^{(i)}+ \delta F(x^{(i)}) + \delta G_i(\textbf{x})+ \sqrt{2 \delta T}Z\right)\right] \leq \kappa(x^{(i)}) -\frac{a}{2} \delta T.
\end{equation}
ii) For all $i \in \llbracket 1,N\rrbracket$ and for all $\textbf{x}\in \mathbb{R}^{dN}$ such that $\lvert x^{(i)} \rvert \geq R$, we have
\begin{equation}\label{eq:bound_L_system}
    \mathbb{E}\left[\kappa \left(x^{(i)}+ \delta F(x^{(i)})+ \delta G_i(\textbf{x}) + \sqrt{2 \delta T}Z\right)\right] \leq \kappa(x^{(i)}) + \frac{3}{2}L \delta T.
\end{equation}
\end{proposition}

\begin{proof}
    We fix $i \in \llbracket 1,N\rrbracket$. Let $\textbf{x}\in \mathbb{R}^{dN}$ with $\lvert x^{(i)} \rvert \leq R$. Then, as a consequence of Assumptions \ref{hypo: F} and \ref{hypo: G},
    \begin{align*}
        \lvert \delta F(x^{(i)}) +\delta G_i(\textbf{x})\rvert &\leq \delta\lvert F(x^{(i)})\rvert + \delta \lvert  G_i(\textbf{x})\rvert\\
        & \leq \delta \left(\sup_{\lvert y^{(i)}\rvert \leq R} \lvert F(y^{(i)})\rvert + M_G+ M_G\lvert x^{(i)} \rvert^p \right)\\
        & \leq \delta \left(\sup_{\lvert y^{(i)}\rvert \leq R} \lvert F(y^{(i)})\rvert + M_G+ M_G R^p \right).
    \end{align*}
Due to the mean value theorem, we have
    \begin{eqnarray*}
        \lefteqn{\mathbb{E}\left[\kappa \left(x^{(i)}+ \delta F(x^{(i)})+ \delta G_i(\textbf{x}) + \sqrt{2 \delta T}Z\right) \right]}\\ &\leq& \mathbb{E}\left[\kappa \left(x^{(i)}+ \sqrt{2\delta T}Z \right)\right] + \delta \lVert \nabla \kappa \rVert_\infty \left(\sup_{\lvert y^{(i)}\rvert \leq R} \lvert F(y^{(i)})\rvert + M_G+ M_G R^p \right)\\
        & \leq& \kappa(x^{(i)}) - a \delta T 
        + \delta \lVert \nabla \kappa \rVert_\infty \left(\sup_{\lvert y^{(i)}\rvert \leq R} \lvert F(y^{(i)})\rvert + M_G+ M_G R^p \right)\\
        & \leq& \kappa \left (x^{(i)} \right) - \frac{a}{2} \delta T,
    \end{eqnarray*}
where the second inequality holds due to condition \eqref{eq:condition_bound_a} in Proposition \ref{prop:existence_kappa}, and the last inequality holds because $T \geq \tilde{T}_1$. Hence we have \eqref{eq:bound_a_system}. Now we fix $i \in \llbracket 1,N\rrbracket$ and let $\textbf{x}$ be any point in $\mathbb{R}^{dN}$ such that $\lvert x^{(i)}\rvert\geq R$. Before showing \eqref{eq:bound_L_system}, we first prove the following implication:
\begin{equation}\label{eq: ith_norm_threshold}
    \lvert x^{(i)}\rvert \geq \tilde{R}  \Longrightarrow \lvert x^{(i)}+ \delta F(x^{(i)})+ \delta G_i(\textbf{x})\rvert  \geq R_*.
\end{equation}
Indeed, taking $\textbf{x}$ such that $\lvert x^{(i)}\rvert \geq \tilde{R}$, we have
\begin{align*}
     \lvert x^{(i)}+ \delta F(x^{(i)})+ \delta G_i(\textbf{x})\rvert &\geq \lvert x^{(i)}\rvert - \delta \lvert  F(x^{(i)})\rvert - \delta \lvert G_i(\textbf{x})\rvert\\
     &\geq \lvert x^{(i)}\rvert -\delta (L_F \lvert x^{(i)}\rvert+ F(0)) - \delta M_G(1+\lvert x^{(i)}\rvert^p)\\
     &\geq \lvert x^{(i)}\rvert -\delta (L_F \lvert x^{(i)}\rvert+ F(0)) - \delta M_G(1+\lvert x^{(i)}\rvert)\\
     &\geq (1-\delta L_F-\delta M_G)\lvert x^{(i)}\rvert -\delta F(0) -
     \delta M_G \geq R_*.
\end{align*}
 We will split the remaining part of the proof into two cases: when $\lvert x^{(i)}\rvert \geq R$, then $\lvert x^{(i)} \rvert$ is either in the range between $R$ and $\tilde{R}$, or greater than $\tilde{R}$. 
 
 Taking $\textbf{x}$ such that $R \leq \lvert x^{(i)} \rvert \leq \tilde{R}$, the term $\lvert \delta F(x^{(i)}) +\delta G_i(\textbf{x})\rvert$ is still bounded, so the proof follows the same argument as for showing \eqref{eq:bound_a_system}. We have
\begin{align*}
     \mathbb{E}&\left[\kappa \left(x^{(i)}+ \delta F(x^{(i)})+ \delta G_i(\textbf{x}) + \sqrt{2 \delta T}Z\right)\right]  \\
     &\leq \mathbb{E}\left[\kappa \left(x^{(i)}+ \sqrt{2\delta T}Z \right)\right] + \delta \lVert \nabla \kappa \rVert_\infty \left(\sup_{\lvert y^{(i)}\rvert \leq \tilde{R}} \lvert F(y^{(i)})\rvert +M_G + M_G \tilde{R}^p \right)\\
     & \leq \kappa(x^{(i)}) + L \delta T + \delta \lVert \nabla \kappa \rVert_\infty \left(\sup_{\lvert y^{(i)}\rvert \leq \tilde{R}} \lvert F(y^{(i)})\rvert +M_G + M_G \tilde{R}^p \right)\\
     & \leq \kappa\left ( x^{(i)}\right) + \frac{3}{2}L \delta T,
\end{align*}
where the last inequality holds because $T \geq \tilde{T}_2$. 

Now, if we take $\textbf{x}$ such that $\lvert x^{(i)} \rvert \geq \tilde{R}$, then thanks to \eqref{eq: ith_norm_threshold} and condition \eqref{eq:condition_bound_L} in Proposition \ref{prop:existence_kappa}, we have
\begin{align*}
    \mathbb{E}&\left[\kappa \left(x^{(i)}+ \delta F(x^{(i)})+ \delta G_i(\textbf{x}) + \sqrt{2 \delta T}Z\right)\right]\\
    &\leq \kappa\left(x^{(i)} +\delta F(x^{(i)}) + \delta G_i(\textbf{x})\right) +L\delta T\\
    &= \kappa \left(x^{(i)} \right) + L \delta T,
\end{align*}
where the last equality holds because $\kappa$ is $0$ at both points. 

\end{proof}

\begin{proof}[Proof of Theorem \ref{thm:rhoContractionForSystem}]
     We will show the one step contraction with respect to $\tilde{\rho}$ defined by \eqref{eq: defRhoTilde}. Let $\textbf{x},\textbf{y}$ be any two points in $\mathbb{R}^{dN}$. Let $(\textbf{X}_n,\textbf{Y}_n)_{\{n\geq 0\}}$ be the synchronous coupling of \eqref{eq: descrete_ststem} issued from $(\textbf{x},\textbf{y})$. We have
    \begin{align*}
        \mathbb{E} [\rho(\textbf{X}_1,\textbf{Y}_1)]&= \mathbb{E} \left[\sum_{i=1}^N \left \lvert X^i_1 - Y^i_1 \right \rvert^2 \left(T + \kappa(X^i_1) +\kappa(Y^i_1)\right)\right]\\
        & =\sum_{i=1}^N \lvert x^{(i)} - y^{(i)} + \delta (F(x^{(i)}) -F(y^{(i)})) +\delta(G_i(\textbf{x}) -G_i(\textbf{y}))\rvert^2 \times\\
        & \times \mathbb{E} \left[\left(T + \kappa(X^i_1) +\kappa(Y^i_1)\right)\right]\\
        & =\underbrace{\sum_{i=1}^N \left(\lvert x^{(i)} - y^{(i)} \rvert^2 +\delta^2 \lvert F(x^{(i)}) -F(y^{(i)})\rvert^2 \right) \mathbb{E}[T+\kappa(X_1^i) +\kappa(Y_1^i)]}_{:=(i)}\\
        &+\underbrace{\sum_{i=1}^N 2\delta(x^{(i)} -y^{(i)}) \cdot (F(x^{(i)})-F(y^{(i)}))\mathbb{E}[T+\kappa(X_1^i) +\kappa(Y_1^i)]}_{:=(ii)}\\
        &+\underbrace{\sum_{i=1}^N \delta^2 \lvert G_i(\textbf{x}) - G_i(\textbf{y})\rvert^2\mathbb{E}[T+\kappa(X_1^i) +\kappa(Y_1^i)]}_{:=(iii)}\\
        &+\underbrace{\sum_{i=1}^N 2 \delta (x^{(i)}-y^{(i)})\cdot (G_i(\textbf{x})-G_i(\textbf{y})) \mathbb{E}[T+\kappa(X_1^i) +\kappa(Y_1^i)]}_{:=(iv)}\\
        &+\underbrace{\sum_{i=1}^N 2 \delta^2 (F(x^{(i)})-F(y^{(i)}))\cdot (G_i(\textbf{x})-G_i(\textbf{y}))\mathbb{E}[T+\kappa(X_1^i) +\kappa(Y_1^i)]}_{:=(v)}
    \end{align*}
We define $c^{\textbf{x},\textbf{y}}_i:= -c$ if $\lvert x^{(i)} \rvert\geq R$ or $\lvert y^{(i)} \rvert\geq R$, and $c^{\textbf{x},\textbf{y}}_i:= K$ otherwise. By Assumption~\ref{hypo: F}, we have
\begin{equation*}
    (i)+(ii) \leq \sum_{i=1}^N (1+\delta^2 L_F^2+2 \delta c^{\textbf{x},\textbf{y}}_i)\lvert x^{(i)} -y^{(i)} \rvert^2\mathbb{E}[T+\kappa(X_1^i) +\kappa(Y_1^i)].
\end{equation*}
For the item $(iii)$, using condition \eqref{eq: hypG1} from Assumption \ref{hypo: G}, we obtain
\begin{align*}
    (iii) &\leq \sum_{i=1}^N \delta^2 \lvert G_i(\textbf{x})-G_i(\textbf{y}) \rvert^2 (T+2\lVert \kappa \rVert_\infty)\\
    &=\delta^2(T+2\lVert \kappa \rVert_\infty) \lvert G(\textbf{x}) -G(\textbf{y})\rvert^2\\
    & \leq \delta^2 L_G^2 (T+2\lVert \kappa \rVert_\infty) \lvert \textbf{x} -\textbf{y}\rvert^2\\
    & = \sum_{i=1}^N \delta^2 L_G^2 (T+2\lVert \kappa \rVert_\infty)\lvert x^{(i)} -y^{(i)}\rvert^2.
\end{align*}
For the items $(iv)$ and $(v)$, we need condition \eqref{eq: hypG2} from Assumption \ref{hypo: G}, as well as Assumption \ref{hypo: F}. We have
\begin{align*}
    (iv) &\leq \sum_{i=1}^N 2 \delta \lvert x^{(i)} - y^{(i)}\rvert \lvert G_i(\textbf{x}) - G_i(\textbf{y}) \rvert(T+2\lVert \kappa \rVert_\infty )\\
    & \leq 2 \delta (T+2\lVert \kappa \rVert_\infty)C_G \lvert \textbf{x}-\textbf{y}\rvert^2\\
    & = \sum_{i=1}^N 2 \delta (T+2\lVert \kappa \rVert_\infty)C_G \lvert x^{(i)}-y^{(i)}\rvert^2
\end{align*}
and
\begin{align*}
    (v) &\leq \sum_{i=1}^N 2 \delta^2 \lvert F(x^{(i)}) - F(y^{(i)})\rvert \lvert G_i(\textbf{x}) - G_i(\textbf{y})\rvert (T+2\lVert \kappa \rVert_\infty)\\
    &\leq \sum_{i=1}^N 2 \delta^2L_F(T+2\lVert \kappa \rVert_\infty) \lvert x^{(i)} - y^{(i)}\rvert \lvert G_i(\textbf{x}) - G_i(\textbf{y})\rvert \\
    & \leq \sum_{i=1}^N 2 \delta^2L_F C_G (T+2\lVert \kappa \rVert_\infty) \lvert x^{(i)} - y^{(i)}\rvert^2.
\end{align*}
With all those bounds, coming back to our main computation, we have
\begin{align*}
    \mathbb{E} [\rho(\textbf{X}_1,\textbf{Y}_1)]& \leq \sum_{i=1}^N (1+ \delta^2 L_F^2 +2 \delta c^{\textbf{x},\textbf{y}}_i) \lvert x^{(i)} - y^{(i)} \rvert^2 \mathbb{E} \left[T + \kappa(X^i_1) +\kappa(Y^i_1)\right]\\
    &+\sum_{i=1}^N (\delta^2 L_G^2 +2 \delta C_G + 2\delta^2L_F C_G)(T+2\lVert \kappa \rVert_\infty) \lvert x^{(i)} -y^{(i)} \rvert^2.
\end{align*}
By Proposition \ref{prop:kappa_system}, we know that for any $i\in \llbracket 1, N\rrbracket$, if either $\lvert x^{(i)} \rvert\geq R$ or $\lvert y^{(i)} \rvert\geq R$, then we have
\begin{align*}
&(1+ \delta^2 L_F^2 -2 \delta c)\lvert x^{(i)} - y^{(i)} \rvert^2 \mathbb{E}[T+\kappa(X^i_1) +\kappa(Y^i_1)]\\
&\leq (1+ \delta^2 L_F^2 -2 \delta c) \lvert x^{(i)} - y^{(i)} \rvert^2 (T + \kappa(x^{(i)}) + \kappa(y^{(i)}) + 3L \delta T)\\
& =  \lvert x^{(i)} - y^{(i)} \rvert^2 (T + \kappa(x^{(i)}) + \kappa(y^{(i)}))(1+ \delta^2 L_F^2 -2 \delta c)\left(1+ \frac{3L \delta T}{T + \kappa(x^{(i)}) + \kappa(y^{(i)})}\right).
\end{align*}
On the other hand, when both $\lvert x^{(i)} \rvert\leq R$ and $\lvert x^{(i)} \rvert\leq R$, we have 
\begin{align*}
 &(1+ \delta^2 L_F^2 +2 \delta K)\lvert x^{(i)} - y^{(i)} \rvert^2  \mathbb{E}[T+\kappa(X^i_1) +\kappa(Y^i_1)]\\
&\leq (1+ \delta^2 L_F^2 +2 \delta K) \lvert x^{(i)} - y^{(i)} \rvert^2 (T + \kappa(x^{(i)}) + \kappa(y^{(i)}) -a \delta T)\\
&=\lvert x^{(i)} - y^{(i)} \rvert^2 (T + \kappa(x^{(i)}) + \kappa(y^{(i)}))(1+ \delta^2 L_F^2 +2 \delta K) \left( 1- \frac{a \delta T}{T+\kappa(x^{(i)}) +\kappa(y^{(i)})}\right).
\end{align*}
Note that since $\kappa$ is non-negative, we can bound $3L\delta T/(T + \kappa(x^{(i)}) + \kappa(y^{(i)}))$ by $3 L \delta$. Moreover, since $T\geq\tilde{T}_3$, we can bound $(-a \delta T)/(T + \kappa(x^{(i)}) + \kappa(y^{(i)}))$ by $(-a \delta)/2$. Following the same argument as in Section \ref{sec: proofForOneParticle} (going from \eqref{eq:to_modify_L} and \eqref{eq:to_modify_a} to \eqref{eq:h}) we get that, when $\delta < \min(\tilde{\delta}_2,\tilde{\delta}_3)$,
\begin{equation*}
    (1+ \delta^2 L_F^2 -2 \delta c)\left(1+ \frac{3L \delta T}{T + \kappa(x^{(i)}) + \kappa(y^{(i)})}\right) \leq 1-h\delta
\end{equation*}
and
\begin{equation*}
    (1+ \delta^2 L_F^2 +2 \delta K) \left( 1- \frac{a \delta T}{T+\kappa(x^{(i)}) +\kappa(y^{(i)})}\right) \leq 1-h\delta,
\end{equation*}
with $h =\min(c/2,a/4)$. As a consequence, we have
\begin{eqnarray*}
    \lefteqn{\mathbb{E} [\rho(\textbf{X}_1,\textbf{Y}_1)]}\\
    & \leq &\sum_{i=1}^N ( 1-h\delta)\lvert x^{(i)} -y^{(i)} \rvert ^2 (T + \kappa(x^{(i)}) + \kappa(y^{(i)}))\\
    & & + \sum_{i=1}^N \left((\delta^2 L_G^2 +2 \delta C_G + 2\delta^2 L_F C_G)(T+2\lVert \kappa \rVert_\infty) \right) \lvert x^{(i)} -y^{(i)} \rvert ^2\\
    & \leq & \sum_{i=1}^N (1-h\delta)\lvert x^{(i)} -y^{(i)} \rvert ^2 (T +\kappa(x^{(i)})+\kappa(y^{(i)}))\\
    & & + \sum_{i=1}^N\frac{(\delta^2 L_G^2 +2 \delta C_G + 2\delta^2L_F C_G)(T+2\lVert \kappa \rVert_\infty)}{T +\kappa(x^{(i)})+\kappa(y^{(i)})}\lvert x^{(i)} -y^{(i)} \rvert ^2 (T +\kappa(x^{(i)})+\kappa(y^{(i)}))\\
    & \leq & \left(1-\left(h - \tilde{h}\right)\delta\right) \rho(\textbf{x},\textbf{y}),
\end{eqnarray*}
where $\tilde{h}$ is defined by \eqref{eq: defhTilde}. Denoting $\tilde{w}:=\left(h-\tilde{h}\right)\delta$, we conclude that 
$$\mathbb{E}[\tilde{\rho}(\textbf{X}_1,\textbf{Y}_1)] \leq (1-\tilde{w})\tilde{\rho}(\textbf{x},\textbf{y}) \quad \text{ for all } \textbf{x},\textbf{y} \in \mathbb{R}^{dN}.$$
Applying \cite[Lemma 2.1]{EberleMajka2019} we get the desired result.
\end{proof}

\begin{remark}\label{remark:rhoContractionForSystem}
    Note that $\min(c/2, L_F/2) \geq 5 C_G$ is a sufficient condition for $h>\tilde{h}$. As a matter of fact, under this condition, we have
    \begin{align*}
        h- \tilde{h} &\geq 5C_G - 2C_G\left(1 +\frac{2\lVert \kappa \rVert_\infty}{T}\right) - (L_G^2 +2 L_F L_G)\left(1 +\frac{2\lVert \kappa \rVert_\infty}{T}\right)\delta\\
        &\geq C_G -2(L_G^2+2L_F L_G)\delta\\
    &\geq \frac{C_G}{2}>0
    \end{align*}
where we get the second inequality by using the lower bound $T \geq \tilde{T}_3$.
\end{remark}

\section{Proof of Theorem \ref{thm:PoincareInequality} and Corollary \ref{cor:Poincare}}\label{sec: proof_local_Poincare}
\begin{proof}[Proof of Theorem \ref{thm:PoincareInequality}]
We have
\begin{align*}
    \bar{Q}^k f^2 - (\bar{Q}^k f)^2 &= \sum_{j=0}^{k-1} \bar{Q}^{k-j}(\bar{Q}^j f)^2 - \bar{Q}^{k-j-1}(\bar{Q}^{j+1}f)^2\\
    &=\sum_{j=0}^{k-1} \bar{Q}^{k-j-1}\left( \bar{Q}(\bar{Q}^j f)^2 - (\bar{Q}(\bar{Q}^j f))^2 \right).
\end{align*}
Applying the local Poincar\'{e} inequality \eqref{eq: localPoincare} with the smooth function $g_j:= \bar{Q}^j f$, we obtain
\begin{align*}
     \bar{Q}^k f^2 - (\bar{Q}^k f)^2  \leq \sum_{j=0}^{k-1} \bar{Q}^{k-j-1} C_{LP} \bar{Q} \lvert \nabla \bar{Q}^j f\rvert ^2.
\end{align*}
In order to proceed, let us observe that thanks to \cite[Theorem 2.2]{Kuwada1}, by choosing $\tilde{d}_2^k(x,y) = \bar{M} (1-\bar{w})^k d_2(x,y)$ in \cite{Kuwada1}, the Wasserstein contraction condition \eqref{eq: W2ContractionGeneral} implies that for all $f$,
\begin{equation*}
    \left \lvert  \nabla \bar{Q}^k f \right \rvert^2 (x) \leq \bar{M}^2 (1-\bar{w})^{2k} \bar{Q}^k \left \lvert \nabla f \right \rvert^2(x). 
\end{equation*}
As a consequence, we obtain
\begin{align*}
    \bar{Q}^k f^2 - (\bar{Q}^k f)^2 & \leq \sum_{j=0}^{k-1} \bar{Q}^{k-j-1} C_{LP} \bar{Q} \Bar{M}^2(1-\bar{w})^{2j} \bar{Q}^j \lvert \nabla f \rvert^2\\
    & =  C_{LP}  \bar{M}^2 \bar{Q}^k \lvert \nabla f \rvert^2 \sum_{j=0}^{k-1} (1-\bar{w})^{2j}  \\
    & \leq C_{LP}  \bar{M}^2 \bar{Q}^k \lvert \nabla f \rvert^2 \sum_{j=0}^\infty (1-\bar{w})^{2j} \\
    & = \frac{C_{LP} \bar{M}^2}{1-(1-\bar{w})^2} \bar{Q}^k \lvert \nabla f \rvert^2.
\end{align*}
Thus we get
\begin{equation}\label{eq: local_Poincare_Q_k}
    \bar{Q}^k f^2(x) - (\bar{Q}^k f)^2(x) \leq C_P \bar{Q}^k \lvert \nabla f \rvert^2(x), \quad \text{ for all } x \in \mathbb{R}^d \text{ and for all } k \in \mathbb{N},
\end{equation}
which is the the local Poincar\'{e} inequality for $\bar{Q}^k$ with constant $C_P =\frac{C_{LP} \bar{M}^2}{1-(1-\bar{w})^2}$. It remains to demonstrate that the invariant measure of the Markov chain satisfies the Poincar\'{e} inequality with the same constant $C_P$. First, it is a well-known fact that the $L^2$-Wasserstein contraction \eqref{eq: W2ContractionGeneral} implies the existence and uniqueness of the invariant measure $\pi_\infty$, due to a fixed point argument. In particular, we know that for any $x \in \mathbb{R}^d$, we have $\bar{Q}^k(x, \cdot) \to \pi_{\infty}$ as $k \to \infty$, in the sense of weak convergence of measures. Hence, by letting $k$ tend to infinity in \eqref{eq: local_Poincare_Q_k}, we obtain the Poincar\'{e} inequality for the invariant measure $\pi_\infty$.
\end{proof}

\begin{proof}[Proof of Corollary \ref{cor:Poincare}]
    Thanks to Corollary \ref{cor: WContraction}, we have the $\mathcal{W}_2$-contraction for $Q$ with constants $\bar{M} = M$ and $\bar{w} = \delta h$.
    Moreover, it is well-known that $Q$ satisfies the local Poincar\'{e} inequality. Specifically, for any $x\in \mathbb{R}^d$, the distribution of $Q(x,\cdot)$ is the Gaussian distribution $\mathcal{N}(x+\delta b(x),2\delta T I_d)$, which satisfies the Poincar\'{e} inequality (see 
     \cite[Proposition 4.8.1]{BakryGentilLedoux}) with the constant given by the largest eigenvalue of the covariance matrix. In our case, this constant is $2\delta T$. Hence we obtain \eqref{eq: localPoincare} with constant $C_{LP} = 2\delta T$. Therefore, based on Theorem \ref{thm:PoincareInequality}, we obtain the desired result.
\end{proof}

\section{Confidence Intervals}\label{sec:ConfidenceIntervals}
This section is concerned with concentration inequalities   for Monte Carlo averages associated to the Markov chain \eqref{eq:chain}. Given two probability measures $\nu, \mu$ on $\mathbb{R}^d$, the relative entropy (or the Kullback-Leibler divergence) of $\nu$ with respect to $\mu$ is defined as
\begin{equation*}
    \mathcal{H}(\nu|\mu) = \begin{cases}
        \int_{\mathbb{R}^d} \log \frac{d \nu}{ d\mu} d\nu \quad &\text{if } \nu \ll \mu\\
        + \infty \quad &\text{otherwise.}
    \end{cases}
\end{equation*}
A measure $\mu$ is said to satisfy the $L^1$-transportation cost-information inequality with
constant $C > 0$ (in which case we write that $\mu$ satisfies $T_1(C)$) if
\begin{equation*}
    \mathcal{W}_1(\nu,\mu) \leq \sqrt{2C \mathcal{H}(\nu|\mu)}.
\end{equation*}

It is known (see e.g.\ \cite[Corollary 2.6]{Djellout}) that, as stated in the next proposition, such an inequality together with a $L^1$-Wasserstein contraction implies Gaussian concentration inequalities for empirical averages (hence non-asymptotic confidence intervals for MCMC estimators). The $L^2$ result obtained in Corollary~\ref{cor: WContraction} implies the weaker $L^1$ one. Of course $L^1$-Wasserstein contraction were already known under Assumptions~\ref{hypo:1} and \ref{hypo:2}, the interest of our new $L^2$ result is that, in some cases, the contraction rate scales well with dimension, as discussed in \cite{moiCourbureTemperature}. Since the statement below is not exactly the same as the one of \cite{Djellout}, we include a detailed proof  for completeness.
\begin{proposition}\label{prop:concentration}
    Let $Q$ be a Markov transition kernel operator on $\mathbb{R}^d$. Suppose that there exist $M \geq 1 >\theta >0$ such that for all $\mu,\nu\in\mathcal{P}(\mathbb{R}^d)$ and for all $k \in \mathbb{N}$ we have 
    \begin{equation}\label{eq:W1contract}
        \mathcal{W}_1(\mu Q^k, \nu Q^k) \leq M(1-\theta)^k \mathcal{W}_1(\mu,\nu).
    \end{equation}
    Suppose that $Q$ satisfies the local $T_1(C)$ on $\mathbb{R}^d$ with a constant $C > 0$, in the sense that for all $x \in \mathbb{R}^d$ and for all measures $\mu \in \mathcal{P}(\mathbb{R}^d)$ we have 
    \begin{equation*}
        \mathcal{W}_1(Q(x,\cdot), \mu) \leq \sqrt{2C \mathcal{H}(\mu | Q(x,\cdot))}.
    \end{equation*}
    Then for the Markov chain $(X_k)_{k \geq 0}$ associated to $Q$ with $X_0 \sim \nu_0$ where $\nu_0$ is any measure satisfying $T_1(C_0)$ inequality with constant $C_0>0$, we have for any 1-Lipschitz $\phi$
    \begin{equation}\label{eq: Confidence_Interval}
        \mathbb{P}\left( \frac{1}{n}\sum_{k=0}^{n-1} \left( \phi(X_k)-\mathbb{E}\left [ \phi(X_k)\right]\right) \geq u \right) \leq \exp \left( - \frac{n^2 u^2\theta^2}{2\left(\left(n-1\right)C \theta^2 + C_0 M^2\right)}\right).
    \end{equation}
\end{proposition}
\begin{proof}
    Thanks to \cite[Theorem 1.1]{Djellout}, we know that if $\mu\in \mathcal{P}(\mathbb{R}^d)$ satisfies $T_1(C')$ with $C'>0$, then for all $\beta$-Lipschitz functions $\varphi : \mathbb{R}^d \to \mathbb{R}$ and for all $\lambda>0$, we have 
    \begin{equation}\label{eq: Concentration_inequality}
        \mu(e^{\lambda \varphi}) \leq e^{C'\beta^2 \lambda^2/2} e^{\lambda \mu(\varphi)}.
    \end{equation}
    Moreover, due to the Wasserstein contraction \eqref{eq:W1contract}, we can see that for any $1$-Lipschitz function $\phi: \mathbb{R}^d \to \mathbb{R}$ and for any $k\in\mathbb{N}$, the function $Q^k \phi$ is $M(1-\theta)^k$-Lipschitz. Indeed, for any coupling $(X_n,Y_n)_{n = 0}^{\infty}$ of two copies of the chain associated with the kernel $Q$, issued from $(x,y) \in \mathbb{R}^{2d}$, we have
    \begin{equation*}
        \lvert Q^k\phi(x) -Q^k\phi(y) \rvert = \lvert \mathbb{E}[\phi(X_k)] -\mathbb{E}[\phi(Y_k)] \rvert \leq \mathbb{E}[\lvert X_k - Y_k\rvert],
    \end{equation*}
    since $\phi$ is $1$-Lipschitz. Taking the infimum over all couplings, we get 
    \begin{equation*}
        \lvert Q^k\phi(x) -Q^k\phi(y) \rvert \leq \mathcal{W}_1(\delta_x Q^k , \delta_y Q^k) \leq M(1-\theta)^k  \mathcal{W}_1(\delta_x , \delta_y) = M(1-\theta)^k \lvert x-y\rvert.
    \end{equation*}
    Suppose $\phi$ is a $1$-Lipschitz function, whereas $\lambda > 0$ and $t \in \mathbb{R}$ are constants. We have 
    \begin{equation*}
        \mathbb{P}\left(\frac{1}{n} \sum_{k=0}^{n-1} \phi(X_k) \geq t \right) = \mathbb{P}\left(e^{\lambda\sum_{k=0}^{n-1} \phi(X_k)} \geq e^{\lambda n t} \right) \leq e^{-\lambda n t } \mathbb{E}\left[e^{\lambda\sum_{k=0}^{n-1} \phi(X_k)}\right].
    \end{equation*}
    By basic properties of the conditional expectation, we get 
    \begin{equation}
        \mathbb{E}\left[e^{\lambda\sum_{k=0}^{n-1} \phi(X_k)}\right] = \mathbb{E}\left[ \mathbb{E} \left[e^{\lambda\sum_{k=0}^{n-1} \phi(X_k)} | X_{n-2}\right]\right] = \mathbb{E}\left[e^{\lambda\sum_{k=0}^{n-2} \phi(X_k)} Q(e^{\lambda \phi})(X_{n-2})\right].
    \end{equation}
    Since, by assumption, $Q(x,\cdot)$ satisfies $T_1(C)$ for all $x \in \mathbb{R}^d$, we can apply \eqref{eq: Concentration_inequality} and obtain
    \begin{equation*}
         \mathbb{P}\left(\frac{1}{n} \sum_{k=0}^{n-1} \phi(X_k) \geq t \right) \leq e^{-\lambda n t +\frac{C\lambda^2}{2} }\mathbb{E}\left[e^{\lambda\sum_{k=0}^{n-2} \phi(X_k)} e^{\lambda Q\phi(X_{n-2})}\right].
    \end{equation*}
    Note that $e^{\lambda Q\phi(X_{n-2})}$ on the right hand side can be replaced by $e^{\lambda Q^{n-1}\phi(X_{0})}$. Now, repeating this argument for $X_{n-3}$,...,$X_1$, we obtain
    \begin{equation*}
         \mathbb{P}\left(\frac{1}{n} \sum_{k=0}^{n-1} \phi(X_k) \geq t \right) \leq e^{-\lambda n t +\frac{(n-1)C\lambda^2}{2} }\mathbb{E}\left[e^{\lambda\sum_{k=0}^{n-1} Q^k\phi(X_0)}\right] = e^{-\lambda n t +\frac{(n-1)C\lambda^2}{2} } \nu_0(e^{\lambda\sum_{k=0}^{n-1}Q^k \phi}).
    \end{equation*}
    Note that the function $\sum_{k=0}^{n-1}Q^k\phi$ is $\frac{M}{\theta}$-Lipschitz. Indeed, for any $x$, $y \in \mathbb{R}^d$, we have 
    \begin{equation*}
        \left \lvert \sum_{k=0}^{n-1} \left(Q^k\phi(x) - Q^k\phi(y) \right) \right \rvert \leq \sum_{k=0}^{n-1} M(1-\theta)^k\lvert x-y \rvert \leq \sum_{k=0}^{\infty}M(1-\theta)^k\lvert x-y \rvert = \frac{M}{\theta}\lvert x- y\rvert.
    \end{equation*}
    Finally, by using again inequality \eqref{eq: Concentration_inequality} for $\nu_0(e^{\lambda \sum_{k=0}^{n-1}Q^k \phi})$, we get 
    \begin{equation*}
         \mathbb{P}\left(\frac{1}{n} \sum_{k=0}^{n-1} \phi(X_k) \geq t \right) \leq \exp\left(-\lambda n t +\frac{(n-1)C\lambda^2}{2} + \frac{C_0 M^2 \lambda^2}{2\theta^2}\right) \exp\left(\lambda \sum_{k=0}^{n-1}\nu_0(Q^k\phi)\right).
    \end{equation*}
    Substituting $t$ by $u + \frac{1}{n} \sum_{k=0}^{n-1} \nu_0 (Q^k \phi)$ we obtain
    \begin{equation*}
         \mathbb{P}\left( \frac{1}{n}\sum_{k=0}^{n-1} \left( \phi(X_k)-\mathbb{E}\left [ \phi(X_k)\right]\right) \geq u \right) \leq \exp\left(-\lambda n u +\frac{(n-1)C\lambda^2}{2} + \frac{C_0 M^2 \lambda^2}{2\theta^2}\right).
    \end{equation*}
    Note that the above equation holds for any $\lambda>0$, and if we take the infimum over all $\lambda >0$ on the right hand side, we get \eqref{eq: Confidence_Interval}.
\end{proof}

In order to apply Proposition \ref{prop:concentration} in our setting, we need the following observations: 
 \begin{itemize}
 \item The $\mathcal W_1$ contraction \eqref{eq:W1contract} is implied by the $\mathcal W_2$ contraction with the same constants. Indeed, assuming the $\mathcal W_2$ contraction and using Jensen's inequality, for any $x$, $y \in \mathbb{R}^d$ we get $\mathcal{W}_1(\delta_x Q^k, \delta_y Q^k) \leq \mathcal{W}_2(\delta_x Q^k, \delta_y Q^k) \leq M(1-\theta)^k \mathcal{W}_2(\delta_x,\delta_y)$. Since $\mathcal{W}_2(\delta_x,\delta_y) = |x-y| = \mathcal{W}_1(\delta_x,\delta_y)$, a conditioning argument implies the $\mathcal W_1$ contraction for any initial conditions $\mu$ and $\nu$. In particular, in our case, due to Corollary \ref{cor: WContraction}, we have $\theta= h\delta$ for some $h>0$ uniformly in $\delta\in(0,\delta_0]$.
 \item The Markov transition kernel $Q$ is such that $Q(x,\cdot)$ is the Gaussian law with mean $x+\delta b(x)$ and with variance $\delta T I_d$. By \cite[Corollary 5.7.2]{BakryGentilLedoux},  it satisfies a log-Sobolev inequality with constant $C=\delta T$, which in turn implies $T_1(C)$ with the same $C$ thanks to \cite[Theorem 1]{otto2000generalization}.
 \end{itemize}
 We see that our estimates are consistent in terms of their dependence on $\delta$. Indeed, by taking $n = \lceil t/\delta\rceil$ for some $t>0$, we end up with 
 \[\mathbb{P}\po \frac1n\sum_{k=0}^{n-1} \po \phi(X_k) - \mathbb E(\phi(X_k))\pf    \geqslant u\pf  \leqslant   \exp \po - \frac{t u^2 h^2}{ 2(T+C_0 M^2/t)}\pf\,.\]
 Now, to obtain a non-asymptotic confidence interval for the target measure, we need to control the bias of the estimator. First, let $\pi_\infty$ be the unique invariant measure of $Q$. Using the fact that $\sum_{j=0}^{k-1} Q^j \phi$ is $M/\theta$-Lipschitz when $\phi$ is $1$-Lipschitz, we get
 \begin{equation}\label{eq: biasBound}
 \left|\frac1n\sum_{k=0}^{n-1} \mathbb E(\phi(X_k)) - \pi_\infty \phi\right| = \frac1n \left|(\nu_0-\pi_\infty) \sum_{k=0}^{n-1} Q^k\phi \right| \leqslant \frac{M}{n\theta} \mathcal W_1(\nu_0,\pi_\infty)\,.
 \end{equation}
 Next, one needs to take into account the discretisation error of $\pi_\infty$ with respect to the invariant measure of the continuous-time process. Following e.g.\ the case 1 of \cite[Lemma 10]{MonmarcheJournel}, we easily get that, for all initial conditions $\nu$,
 \[\mathcal W_1 \po \nu Q^n,\nu P_{n\delta}\pf \leqslant G(n\delta) \sqrt{\delta} \int_{\R^d}|x| \nu(\dd x)\,,\]
 where $(P_t)_{t\geqslant 0}$ is the semi-group associated to the continuous-time diffusion process and $G$ is some increasing function (that can be made explicit). Applying this with $\nu = \nu_\infty$ the invariant measure of the continuous diffusion and with $n_*=\lceil t/\delta\rceil$ where $t$ is chosen so that, from \eqref{eq:W1contract} (in our case, namely with $\theta= h\delta$)
 \[\mathcal W_1(\nu Q^{n_*},\mu Q^{n_*}) \leqslant \frac12 \mathcal W_1(\nu,\mu) \,,\]
 for all $\nu$, $\mu$, uniformly in $\delta\in (0,\delta_0]$, we get
 \begin{eqnarray*}
 \mathcal W_1(\pi_\infty,\nu_\infty) &\leqslant& \mathcal W_1(\pi_\infty Q^{n_*},\nu_\infty Q^{n_*}) + \mathcal W_1(\nu_\infty Q^{n_*},\nu_\infty P_{n_*\delta})   \\
 &\leqslant& \frac12  \mathcal W_1(\pi_\infty,\nu_\infty) + G(t+\delta_0) \sqrt{\delta} \int_{\R^d}|x| \nu_\infty(\dd x)\,,
 \end{eqnarray*}
 hence
 \begin{equation}\label{eq:discretisationError}
 \mathcal W_1(\pi_\infty,\nu_\infty) \leqslant 2 G(t+\delta_0) \sqrt{\delta} \int_{\R^d}|x| \nu_\infty(\dd x)\,.
 \end{equation}
 As a conclusion, combining the fact that $|\pi_\infty \phi - \nu_\infty \phi| \leq  \mathcal W_1(\pi_\infty,\nu_\infty)$ for any $1$-Lipschitz function $\phi$ with \eqref{eq:discretisationError}, and applying \eqref{eq: biasBound}, the bias can be bounded as
 \[\mathrm{bias} := \left|\frac1n\sum_{k=0}^{n-1} \mathbb E(\phi(X_k)) - \nu_\infty \phi\right| \leqslant \frac{c_1}{n} + c_2\sqrt{\delta}\,,\]
 for some explicit $c_1,c_2$ and we get a non-asymptotic confidence interval by bounding, for $u>0$, 
 \begin{eqnarray}
 \mathbb{P}\po \left|\frac1n\sum_{k=0}^{n-1} \phi(X_k) - \nu_\infty\phi \right|  \geqslant u + \mathrm{bias} \pf & \leqslant & \mathbb{P}\po \left|\frac1n\sum_{k=0}^{n-1} \po \phi(X_k) - \mathbb E(\phi(X_k))\pf    \right| \geqslant u\pf \nonumber\\
 &\leqslant & 2 \exp \po - \frac{t u^2h^2}{2(T + C_0 M^2/t)}\pf\,,\label{eq:boundCI}
 \end{eqnarray}
 where the factor $2$ comes from the fact we applied Proposition~\ref{prop:concentration} to both $\phi$ and $-\phi$.

 Now, if we want to have a confidence interval of length smaller than some $\varepsilon>0$ with a probability larger than some $\alpha<1$, first we choose $\delta$ small enough so that the discretization bias $c_2\sqrt{\delta}$ is smaller than $\varepsilon/3$, and then we choose $n$ large enough so that the long-time convergence bias $c_1/n$ is smaller than $\varepsilon/3$  and that the right-hand side of \eqref{eq:boundCI} with $u=\varepsilon/3$ is smaller than $\alpha$. In other words, the choice of the parameters $n,\delta$ is completely explicit in terms of $\varepsilon$ and $\alpha$.

\section{From \texorpdfstring{$L^2$}{L^2}-Wasserstein to entropy}\label{sec:EntropyRegularization}

	This section is concerned with entropy-cost inequalities in the spirit of  \cite{GuillinWang2011,rockner2010log}, which describe a small-time regularization property of the chain: if the initial condition has a finite second moment, then, for all positive times, it has a finite  KL-divergence with respect to the equilibrium distribution. Combining this with the $L^2$-Wasserstein contraction of Corollary~\ref{cor: WContraction} thus yields a non-asymptotic convergence bound in  KL-divergence.
 
 More specifically, \cite{GuillinWang2011,rockner2010log} are concerned with continuous-time diffusion processes, and we give an analogue of \cite[Corollary 1.2]{rockner2010log} for the Euler scheme \eqref{eq:chain}   (see also \cite{MonmarcheIdealized} for a similar discrete-time result for Hamiltonian Monte Carlo, and the related one-shot coupling of \cite{Bou-RabeeEberle2023}). In particular, consistent with the rest of our work, we will be interested in a result with a correct dependency in the step-size $\delta$, so that we recover \cite[Corollary 1.2]{rockner2010log} as $\delta$ vanishes. As we will see, to do this, we will need to consider $n$ transitions of \eqref{eq:chain} with $n$ of order $\delta^{-1}$. However, since it gives some indication on the general proof and does not require any conditions other than Assumption~\ref{hypo:1}, let us first briefly discuss the case $n=1$.

 In the rest of this section, $Q$ stands for the Markov transition operator of the chain \eqref{eq:chain}. In other words, for $x\in\R^d$, $\delta_x Q$ is an isotropic Gaussian distribution with mean $x+\delta b(x)$ and with variance $2T\delta$. Denoting by $f_x$ its density, we see that
 \begin{eqnarray*}
 \mathcal H(\delta_x Q\ |\ \delta_y Q) &=& \int_{\R^d} f_x(z) \log \frac{f_x(z)}{f_y(z)} dz\\
 &=& \frac{1}{2\delta T} \int_{\R^d} \left(|z-(y+\delta b(y)|^2-|z-(x+\delta b(x)|^ 2\right) f_x(z)\dd z    \\
 & = & \frac{1}{2\delta T} \left(x+\delta b(x) - y - \delta b(y)\right)^2 \,.
 \end{eqnarray*}
 Hence, under Assumption~\ref{hypo:1}, for any $x,y\in\R^d$,
 \begin{equation}
     \label{eq:Hn1local}
     \mathcal H(\delta_x Q\ |\ \delta_y Q) \leqslant \frac{(1+\delta L_b)^2}{2\delta T}|x-y|^2\,.
 \end{equation}
As we will see in the proof of Proposition~\ref{prop:entropy2} below, this implies that for all $\nu \in \mathcal{P}(\mathbb{R}^d)$, denoting by $\pi_\infty$ the invariant measure of $Q$, 
    \begin{equation}\label{eq:Hn=1}
        \mathcal{H}(\nu Q | \pi_\infty) \leq \frac{(1+\delta L_b)^2}{2\delta T} \mathcal{W}_2^2(\nu,\pi_\infty).
    \end{equation}
This result is already valuable to get a convergence rate in  KL-divergence (and hence in total variation thanks to Pinsker's inequality) in the setting of Corollary~\ref{cor: WContraction}, but it degenerates as $\delta$ vanishes (and gives an extra $\ln(\delta^{-1})$ term in the complexity bounds obtained that way). For this reason, in the following, we focus on obtaining  a bound similar to \eqref{eq:Hn1local} but with $Q^n$ for any $n$ instead of $n=1$.

To do this, we will need the following additional assumption on $b$.

\begin{assumption}\label{hypo:Lipschitz_gradient}
    We suppose that there exists $C>0$ such that
    \begin{equation}
         \lvert \nabla b(x) - \nabla b(y)\rvert \leq C\lvert x-y\rvert \quad \text{ for all } x,y \in \mathbb{R}^d.
    \end{equation}
\end{assumption}
Here, $\lvert \nabla b(x) - \nabla b(y) \rvert$ stands for the operator norm (associated to the Euclidean norm on $\R^d)$ of the difference between the Jacobian matrix of $b$ at points $x$ and $y$.

The main part of this section is the proof of the following result.

\begin{proposition}\label{prop: n_step_entropic_regularization}
    Let $Q$ be the transition kernel of the Markov chain \eqref{eq:chain}, and let $c>0$. Under Assumptions \ref{hypo:1} and \ref{hypo:Lipschitz_gradient}, and assuming moreover that
    \begin{equation}\label{eq: delta_entropic_regularization}
        0< \delta \leq \frac{1}{16cL_b^2\exp(2cL_b)}\,,
    \end{equation}
    then for all $n \geq 2$ with  $\delta n <c$ and  all $x, y \in \mathbb{R}^d$,
    \begin{equation}
        \mathcal{H}(\delta_y Q^n| \delta_x Q^n) \leq \left( \frac{1}{2 T} \left(c L_b^2 + \frac{1}{n \delta}\right) + \frac{1}{2} c^2 C^2d \exp(2cL_b)\right)\rvert x- y \rvert^2.
    \end{equation}
    
\end{proposition}

\begin{proof}
    Let $x,y \in \mathbb{R}^d$. We fix an integer $n \geq 2$, and let $\textbf{G}:=(G_1,...,G_n)$ where $(G_i)_{i\in \llbracket 1, n\rrbracket}$ are i.i.d.\ with the standard normal distribution on $\mathbb{R}^d$. We consider two Markov chains with transition kernel $Q$, initialized from $x$ and $y$, respectively, and driven by the same noise $\textbf{G}$.
    \begin{align*}
        X_{k+1}&=X_k+\delta b(X_k)+ \sqrt{2\delta T}G_{k+1}, \quad  k =  0,\cdots, n-1\\
        Y_{k+1}&=Y_k+\delta b(Y_k)+ \sqrt{2\delta T}G_{k+1}, \quad  k = 0,\cdots, n-1
    \end{align*}
    We then consider the function $\Psi_x^n : \mathbb{R}^{dn} \to \mathbb{R}^{dn}$ such that $\Psi_x^n(\textbf{G}):=(\Psi^{1,n}_x(\textbf{G}),...,\Psi^{n,n}_x(\textbf{G})) =(X_1, \ldots, X_n)$. Moreover, we introduce $\Psi_x: \mathbb{R}^{dn} \to \mathbb{R}^{dn}$ given by $\Psi_x(\textbf{G}) = (W_1, \ldots, W_n) =: \textbf{W}$,
    where for all $k \in \llbracket 1, n \rrbracket$,
    \begin{equation}\label{eq: defW}
         W_k = G_k +\frac{1}{\sqrt{2\delta T}}\left(\delta(b(Y_{k-1})-b(\tilde{X}_{k-1}))+\frac{1}{n}(y-x)\right),
    \end{equation}
    with
    \begin{equation}\label{eq: defXtilde}
        \tilde{X}_{k+1}=\tilde{X}_k+\delta b(\tilde{X}_k)+ \sqrt{2\delta T}W_{k+1}
    \end{equation}
    defined for $k \in \llbracket 0, n-1 \rrbracket$. One can check that $\Psi_x^{n,n}(\Psi_x(\textbf{G})) = Y_n$ (by design, $\textbf{W}=\Psi_x(\textbf{G})$ is such that $\Psi_x^{k,n}(\textbf{W}) - Y_k = \Psi_x^{k-1,n}(\textbf{W}) - Y_{k-1} -\frac1n (x-y) = \dots = \frac{n-k}{n}(x-y)$). Moreover, we have $\Psi^{n,n}_x(\textbf{G})=X_n$, i.e., the last marginals of $\Psi^n_x \circ \Psi_x(\textbf{G})$ and $\Psi^n_x(\textbf{G})$ are $\delta_y Q^n$ and $\delta_x Q^n$, respectively. By the entropy decomposition we have
    \begin{equation}\label{eq: entropy_decomposition}
        \mathcal{H}(\delta_y Q^n | \delta_x Q^n) \leq \mathcal{H}\left( \operatorname{Law} \left( \Psi^n_x \circ \Psi_x(\textbf{G})\right)  | \operatorname{Law} \left( \Psi^n_x(\textbf{G}) \right) \right) \,.
    \end{equation}
    Indeed, here we use the fact that if measures $\mu$ and $\nu$ on $\mathbb{R}^d$ are absolutely continuous with respect to the Lebesgue measure, and have $k$-dimensional marginals (with $k < d$) denoted by $\mu_1$ and $\nu_1$, respectively, then $\mathcal{H}(\mu_1|\nu_1) \leq \mathcal{H}(\mu|\nu)$, which itself follows from the following classical  calculation. Denoting $x:=(x_*,x_{**})$ where $x_* \in \mathbb{R}^k$ and $x_{**} \in \mathbb{R}^{d-k}$,
    \begin{equation*}
    \begin{split}
    \mathcal{H}(\mu|\nu) 
    & = \mathcal{H}(\mu_1 | \nu_1) + \int_{\mathbb{R}^k} \mathcal{H}\left( \frac{\mu(x_*,\cdot)}{\mu_1(x_*)} \Bigg| \frac{\nu(x_*,\cdot)}{\nu_1(x_*)}\right) \mu_1(x_*) d x_* \,,
    \end{split}
    \end{equation*}
    and the last integral is non-negative since entropy is non-negative. 
    
    In order to deal with the right hand side of \eqref{eq: entropy_decomposition}, we will use explicit formulas for the densities of $\Psi_x^n \circ \Psi_x(\textbf{G})$ and $\Psi_x^n(\textbf{G})$. Denoting $\varphi(u) = (2\pi)^{-(dn)/2} \exp(-\lvert u \rvert^2/2)$, and using a change of variables, these densities are given by $f$ and $g$ defined, respectively, as
\begin{align*}
    &f(u) = \frac{1}{(2\pi)^{(dn)/2}} \lvert \det \nabla (\Psi^n_x \circ \Psi_x)^{-1}(u)\rvert e^{-\frac{\left\lvert (\Psi^n_x \circ \Psi_x)^{-1}(u) \right\rvert^2}{2}}\\
    &g(u) = \frac{1}{(2\pi)^{(dn)/2}} \lvert \det \nabla (\Psi^n_x)^{-1}(u)\rvert e^{-\frac{\left\lvert (\Psi^n_x)^{-1}(u) \right\rvert^2}{2}}.
\end{align*}
    With the above notation, we have
    \begin{align*}
        \mathcal{H}\left( \operatorname{Law} \left( \Psi^n_x \circ \Psi_x(\textbf{G})\right)  | \operatorname{Law} \left( \Psi^n_x(\textbf{G}) \right) \right) &= \int_{\mathbb{R}^{dn}} \log \left( \frac{f(u)}{g(u)}\right) f(u) du\\
        & = \mathbb{E}\left [ \log \left( \frac{f(\Psi^n_x\circ\Psi_x(\textbf{G}))}{g(\Psi^n_x\circ\Psi_x(\textbf{G}))}\right)\right]\\
        & = \int_{\mathbb{R}^{dn}} \log \left( \frac{f(\Psi^n_x\circ\Psi_x(u))}{g(\Psi^n_x\circ\Psi_x(u))}\right) \varphi(u) du\\
        & = \int_{\mathbb{R}^{dn}} \log \left( \frac{\varphi(u)}{\frac{\varphi(u) g(\Psi^n_x\circ\Psi_x(u))}{f(\Psi_x^n\circ\Psi_x(u))}}\right) \varphi(u) du \,,
    \end{align*}
    where the function $\frac{\varphi(u) g(\Psi^n_x\circ\Psi_x(u))}{f(\Psi_x^n\circ\Psi_x(u))}$ is the density of $\Psi^{-1}_x(\textbf{G})$.
    Hence we have 
    \begin{equation}\label{eq: entropy_to_entropy}
        \mathcal{H}\left( \operatorname{Law} \left( \Psi^n_x \circ \Psi_x(\textbf{G})\right)  | \operatorname{Law} \left( \Psi^n_x(\textbf{G}) \right) \right) =\mathcal{H}\left( \operatorname{Law} (\textbf{G}) | \operatorname{Law} \left( \Psi_x^{-1}(\textbf{G}) \right) \right).
    \end{equation}
It is proven in \cite[Lemma 15]{Bou-RabeeEberle2023} that, provided
\begin{equation}\label{eq:operator_norme_1/2}
\lvert\nabla\Psi_x(\textbf{G}) -I_{dn} \rvert \leq 1/2
\end{equation}
(recall that $|\cdot|$ stands for the operator norm here) we have
\begin{equation}\label{eq:regularization_Lemma15}
    \mathcal{H}(\textbf{G} | \Psi_x^{-1}(\textbf{G})) \leq \mathbb{E}\left[\frac{1}{2} \lvert \Psi_x(\textbf{G}) -\textbf{G}\rvert^2 + \lVert \nabla \Psi_x(\textbf{G}) - I_{dn}\rVert_F^2\right]\,,
\end{equation}
where $\|\cdot\|_F$ stands for the Fröbenius norm. 
In the following, we will first verify \eqref{eq:operator_norme_1/2}, and then provide bounds for the two terms on the right hand side of \eqref{eq:regularization_Lemma15}.

From the definition \eqref{eq: defW} of $W$, we have $\nabla_{G_j}W_k = \textbf{0}$ for all $1 \leq k < j \leq n$. We also have $\nabla_{G_j}W_k = I_d$ for all $1 \leq k = j \leq n$. Hence we have
\begin{equation}\label{eq: matix_in_Frob}
    \nabla \Psi(\textbf{G})-I_{dn} = \begin{pmatrix}
        \textbf{0} &\nabla_{G_1} W_2 &\cdots & \nabla_{G_1} W_n\\
        \textbf{0} & \ddots& \ddots & \vdots\\
        \vdots &  \ddots & \ddots & \nabla_{G_{n-1}} W_n\\
        \textbf{0}& \cdots & \textbf{0} &\textbf{0} 
    \end{pmatrix}
\end{equation}
where for all $0\leq j< k \leq n$ with $j < k-1$,
\begin{equation}\label{eq:gradient_W}
    \nabla_{G_j}W_k = \frac{\sqrt{\delta}}{ \sqrt{2T}} \left( \nabla b(Y_{k-1}) \nabla_{G_j} Y_{k-1}  - \nabla b (\tilde{X}_{k-1}) \nabla_{G_j} \tilde{X}_{k-1}\right).
\end{equation}
Moreover, for $j = k-1$
\begin{equation*}
    \nabla_{G_j}Y_k =\nabla_{G_j}\tilde{X}_k = \sqrt{2\delta T}I_d,
\end{equation*}
and hence, due to Assumption \ref{hypo:1},
\begin{equation*}
    \left\lvert \nabla_{G_{k-1}} W_{k} \right\rvert = \delta \left\lvert \nabla b (Y_{k-1}) - \nabla b(\tilde{X}_{k-1})\right\rvert \leq 2 \delta L_b.
\end{equation*}
For $j< k-1$, we have
\begin{align*}
    \nabla_{G_j} Y_{k-1} &=(I_d + \delta \nabla b (Y_{k-2}))\nabla_{G_j} Y_{k-2}\\
    &=(I_d + \delta \nabla b (Y_{k-2}))(I_d + \delta \nabla b (Y_{k-3}))\nabla_{G_j} Y_{k-3}\\
    &=\ldots =  A_{k-2} \cdots A_{j+1} A_j \nabla_{G_j} Y_j,
\end{align*}
with $A_r := I_d + \delta \nabla b(Y_r)$. Using $\lvert A_r\rvert \leq 1+ \delta L_b$ for all $r \geq 1$, and $\nabla_{G_j}Y_j = \sqrt{2\delta T}I_d$, we end up with
\begin{equation}\label{eq: nablaYbound}
    \left \lvert \nabla_{G_j} Y_{k-1}\right \rvert \leq \sqrt{2\delta T}(1+\delta L_b)^{k-1-j}.
\end{equation}
Observe that, by \eqref{eq: defW} and \eqref{eq: defXtilde}, $Y_i - \tilde{X}_i = \left(1-\frac{i}{n}\right)(y-x)$. Hence we have 
\begin{equation}\label{eq: equalGradients}
     \nabla_{G_j} \tilde{X}_{k-1} = \nabla_{G_j} Y_{k-1}.
\end{equation}
Using \eqref{eq: equalGradients} and plugging \eqref{eq: nablaYbound} into \eqref{eq:gradient_W}, we get
\begin{align*}
    \lvert \nabla_{G_j} W_k\rvert &\leq \frac{\sqrt{\delta}}{\sqrt{2T}} \left(  \left\lvert \nabla b(Y_{k-1})\right \rvert \left\lvert \nabla_{G_j}Y_{k-1}\right \rvert +\left\lvert \nabla b(\tilde{X}_{k-1})\right \rvert 
 \left\lvert \nabla_{G_j}\tilde{X}_{k-1}\right \rvert \right)\\
    &\leq \frac{\sqrt{\delta}}{\sqrt{2T}} 2L_b \left\lvert \nabla_{G_j}Y_{k-1}\right \rvert\\
    &\leq 2 \delta L_b (1+ \delta L_b)^{k-1-j}.
\end{align*}
We can now conclude with
\begin{equation*}
    \left \lvert \nabla \Psi(\textbf{G}) -I_{dn}\right \rvert^2 \leq \sup_{1\leq k \leq n} \sum_{j=1}^{k-1} \left\lvert \nabla_{G_j} W_k \right\rvert^2
    =\sum_{j=1}^{n-1} \left \lvert \nabla_{G_j}W_n\right \rvert^2
     \leq 4n\delta^2 L_b^2 (1+\delta L_b)^{2n}.
\end{equation*}
Using $\log(1+x) \leq x$ for all $x>0$, then using $0<n \delta<c$, we get
\begin{equation}
    \left \lvert \nabla \Psi(\textbf{G}) -I_{dn}\right \rvert^2 \leq 4 n \delta^2 L_b^2 \exp(2n\delta L_b) \leq 4c\delta L_b^2 \exp(2cL_b)\leq \frac{1}{4},
\end{equation}
since we have $\delta <1/(16cL_b^2\exp(2cL_b))$. Hence we arrive at \eqref{eq:operator_norme_1/2}.

As explained above, by \cite[Lemma 15]{Bou-RabeeEberle2023} this gives  \eqref{eq:regularization_Lemma15}, and we now focus on estimating $\lvert \Psi_x(\textbf{G}) -\textbf{G}\rvert$. We have
\begin{align*}
    \lvert W_k - G_k \rvert &= \frac{1}{\sqrt{2\delta T}}\left \rvert \delta (b(Y_{k-1}) - b(\tilde{X}_{k-1})) + \frac{1}{n}(y-x)\right \rvert\\
    & \leq \frac{1}{\sqrt{2\delta T}} \left(\delta L_b \left \lvert  Y_{k-1} - \tilde{X}_{k-1}\right \rvert +\frac{1}{n}\lvert y-x \rvert\right)\\
    & \leq \frac{1}{\sqrt{2 \delta T}} \left( \delta L_b + \frac{1}{n}\right) \lvert x-y \rvert,
\end{align*}
and hence
\begin{align*}
    \lvert \textbf{W} - \textbf{G} \rvert^2 &\leq 
    \frac{1}{2 \delta T} \sup_{1\leq k \leq n} \sum_{j=1}^k \left( \delta L_b + \frac{1}{n}\right)^2 \lvert x-y \rvert^2\\
    & \leq \frac{1}{\delta T} \left ( n\delta^2 L_b^2 + \frac{1}{n}\right) \lvert x-y\rvert^2.
\end{align*}
Thus we have 
\begin{equation}\label{eq:PsiGG2}
    \frac{1}{2} \lvert \Psi(\textbf{G}) - \textbf{G}\rvert^2 = \frac{1}{2} \lvert \textbf{W} - \textbf{G}\rvert^2 \leq \frac{1}{2 T}\left(n \delta L_b^2 + \frac{1}{n \delta} \right) \lvert x- y \rvert^2.
\end{equation}

Regarding the estimation of $\lVert \nabla \Psi_x(\textbf{G}) - I_{dn}\rVert_F$, using \eqref{eq: equalGradients} we have
\begin{align*}
    \left \lVert \nabla_{G_j} W_k\right \rVert_F &= \left \lVert \frac{\sqrt{\delta}}{\sqrt{2 T}} \left( \nabla b (Y_{k-1}) \nabla_{G_j}Y_{k-1} - \nabla b (\tilde{X}_{k-1}) \nabla_{G_j}\tilde{X}_{k-1}\right)\right \rVert_F \\
    &\leq \frac{\sqrt{\delta}}{\sqrt{2 T}} \left \lvert \nabla b (Y_{k-1}) - \nabla b (\tilde{X}_{k-1})\right \rvert \left \lVert \nabla_{G_j}Y_{k-1} \right\rVert_F\\
    & \leq \frac{\sqrt{\delta}}{\sqrt{2 T}} C\left \lvert Y_{k-1} - \tilde{X}_{k-1}\right \rvert \sqrt{2\delta T} (1+\delta L_b)^{k-1-j}\sqrt{d}\\
    & \leq C \sqrt{d} \delta(1+\delta L_b)^{k-1-j} \lvert x-y \rvert,
\end{align*}
where in the penultimate inequality we use the same argument as for obtaining \eqref{eq: nablaYbound}, together with  $\left \lVert \nabla_{G_{k-1}}Y_k \right \rVert_F \leq \sqrt{2\delta T d}$. We have 
\begin{align*}
    \left \lVert \nabla \Psi(\textbf{G})-I_{dn}\right \rVert_F^2 &= \sum_{1\leq j < k \leq n} \left \lVert  \nabla_{G_j}W_k \right \rVert_F^2\\
    & \leq \sum_{1\leq j < k \leq n} C^2 \delta^2 d (1+ \delta L_b)^{2(k-1-j)} \lvert x-y \rvert^2\\
    & \leq \frac{1}{2}n^2 C^2 \delta^2 d (1+\delta L_b)^{2n} \lvert x-y \rvert^2\\
    & \leq \frac{1}{2}c^2C^2 d\exp(2cL_b) \lvert x-y\rvert^2.
\end{align*}

Plugging this last bound and \eqref{eq:PsiGG2} in \eqref{eq:regularization_Lemma15}, concludes the proof.
\end{proof}

From the last result, the arguments of \cite{rockner2010log} yield the desired entropy-cost regularization inequality, as we now state.

\begin{proposition}\label{prop:entropy2}
Let Assumptions \ref{hypo:1} and \ref{hypo:Lipschitz_gradient} hold. Let $\pi_\infty$ denote the invariant measure of $Q$. Then, for any probability measure $\nu$ on $\R^d$, and for any $n\geq 2$,
\begin{equation*}
    \mathcal{H}(\nu Q^n | \pi_\infty) \leq \left( \frac{1}{2 T}\left(n\delta L_b^2 +\frac{1}{n \delta }\right) + \frac{1}{2} c^2C^2 d \exp(2cL_b)\right) \mathcal{W}_2^2 (\nu,\pi_\infty).
\end{equation*}
\end{proposition}

\begin{proof}
    Let $h: \mathbb{R}^d \to \mathbb{R}$ be positive and bounded, and let $x$, $y \in \mathbb{R}^d$. Denote by $f_{x,n}$ and $f_{y,n}$ the densities of $\delta_x Q^n$ and $\delta_y Q^n$, respectively. We have 
    \begin{align*}
        Q^n \log h (y) &= \int_{\mathbb{R}^d} \log h(u) f_{y,n}(u)du\\
        &= \int_{\mathbb{R}^d} \log h(u) \frac{f_{y,n}(u)}{f_{x,n}(u)}f_{x,n}(u)du\\
        & \leq \log \int_{\mathbb{R}^d}h(u)f_{x,n}(u)du + \int_{\mathbb{R}^d}\log \left(\frac{f_{y,n}(u)}{f_{x,n}(u)} \right) f_{y,n}(u) du\\
        & = \log (Q^n h(x)) + \mathcal{H}(\delta_y Q^n | \delta_x Q^n),
    \end{align*}
    where the inequality above follows from \cite[Lemma 2.4]{ARNAUDON20093653}. Let $\mu$, $\nu \in \mathcal{P}(\mathbb{R}^d)$. Applying the previous inequality with $h = d(\nu Q^n)/d\mu$, i.e., $h=(Q^n)^* h_0$ with $h_0 = d \nu/ d\mu$, and combining this with Proposition \ref{prop: n_step_entropic_regularization} to bound the last term, we obtain
    \begin{equation*}
    \begin{split}
        Q^n (\log ((Q^n)^* h_0))(y) &\leq \log (Q^n (Q^n)^* h_0(x)) \\
        &+\left( \frac{1}{2 T}\left(n\delta L_b^2 +\frac{1}{n \delta }\right) + \frac{1}{2} c^2C^2 d \exp(2cL_b)\right)\lvert x-y \rvert^2.
    \end{split}    
    \end{equation*}
    Integrating this with respect to a coupling $\gamma$ of $\mu$ and $\nu$, we get
    \begin{equation*}
    \begin{split}
        \mathcal{H}(\nu Q^n | \mu) &\leq  \int_{\mathbb{R}^d} \log (Q^n (Q^n)^* h_0) d\mu \\
        &+\left( \frac{1}{2 T}\left(n\delta L_b^2 +\frac{1}{n \delta }\right) + \frac{1}{2} c^2C^2 d \exp(2cL_b)\right) \int_{\mathbb{R}^d} \lvert x-y \rvert^2 d\gamma(x,y),
    \end{split}
    \end{equation*}
    where the first term on the right hand side is $-\mathcal{H}(\mu|\nu Q^n)$, which is non-positive. Hence we have 
    \begin{equation*}
        \mathcal{H}(\nu Q^n | \mu) \leq \left( \frac{1}{2 T}\left(n\delta L_b^2 +\frac{1}{n \delta }\right) + \frac{1}{2} c^2C^2 d \exp(2cL_b)\right) \int_{\mathbb{R}^d} \lvert x-y \rvert^2 d\gamma(x,y).
    \end{equation*}
    Taking the infimum over all couplings concludes the proof.
\end{proof}

\bibliographystyle{plain}
\bibliography{biblio}
\end{document}